\documentclass[reqno,12pt]{amsart}

\usepackage{amsmath, amssymb}
\usepackage[hypertex]{hyperref}
\usepackage{color}

\numberwithin{equation}{section}

\newcommand{\Z}{{\mathbb Z}}
\newcommand{\Q}{{\mathbb Q}}

\newcommand{\C}{{\mathbb C}}

\newcommand{\CC}{{\mathcal C}}

\newcommand{\al}{\alpha}

\newcommand{\la}{\langle}
\newcommand{\ra}{\rangle}

\DeclareMathOperator{\Aut}{Aut}

\DeclareMathOperator{\Hom}{Hom}
\DeclareMathOperator{\iden}{id}

\DeclareMathOperator{\Irr}{Irr}

\DeclareMathOperator{\qdim}{qdim}
\DeclareMathOperator{\Rep}{Rep}

\DeclareMathOperator{\tr}{tr}

\newtheorem{theorem}{Theorem}[section]
\newtheorem{proposition}[theorem]{Proposition}
\newtheorem{lemma}[theorem]{Lemma}
\newtheorem{corollary}[theorem]{Corollary}
\newtheorem{remark}[theorem]{Remark}

\newtheorem{hypothesis}[theorem]{Hypothesis}

\newcommand{\SC}[1]{\Irr(#1)_{\mathrm{sc}}}
\newcommand{\abs}[1]{\lvert{#1}\rvert}

\begin{document}

\title[Simple current extensions  of VOAs]
{Simple current extensions of tensor products of vertex operator algebras}

\author[H. Yamada]{Hiromichi Yamada}
\address{Department of Mathematics, Hitotsubashi University, Kunitachi,
Tokyo 186-8601, Japan}
\email{yamada.h@r.hit-u.ac.jp}

\author[H. Yamauchi]{Hiroshi Yamauchi}
\address{Department of Mathematics, 
Tokyo Woman's Christian University, Suginami-ku, Tokyo 167-8585, Japan}
\email{yamauchi@lab.twcu.ac.jp}



\begin{abstract} 
We study simple current extensions of tensor products of two vertex operator algebras 
satisfying certain conditions. 
We establish the relationship between the fusion rule for the simple current extension and 
the fusion rule for a tensor factor. 
In a special case, 
we construct a chain of simple current extensions. 
We discuss certain irreducible twisted modules for the simple current extension as well.
\end{abstract}

\maketitle

\section{Introduction}\label{sec:introduction}

Simple current extensions of vertex operator algebras have been studied extensively, 
see \cite{Carnahan2014, CKL2015, vEMS2017, Yamauchi2004} 
and the references therein. 
Extensions of vertex operator algebras in more general settings were also investigated, 
see for example \cite{CKM2017, HKL2015, KO2002}. 
In this paper, 
we consider simple current extensions of tensor products of two vertex operator algebras 
with suitable properties. 

We argue under the following setting (Hypothesis \ref{hypo:W_V}). 
Let $W$ and $V$ be simple, self-dual, rational, and $C_2$-cofinite  
vertex operator algebras of CFT-type. 
Assume that all irreducible $V$-modules are simple currents. 
This means that the set of equivalence classes $\Irr(V)$ of irreducible $V$-modules 
is a $C$-graded set $\{ V^\al \mid \alpha \in C \}$ of simple current $V$-modules 
for a finite abelian group $C$ with $V^0 = V$ and 
$V^\alpha \boxtimes_V V^\beta = V^{\alpha+\beta}$ for $\alpha$, $\beta \in C$.  
Let $D$ be a subgroup of $C$ and $\{ W^\beta \mid \beta \in D\}$ a $D$-graded set of 
simple current $W$-modules with $W^0 = W$ and  
$W^\alpha \boxtimes_W W^\beta = W^{\alpha+\beta}$ for $\alpha$, $\beta \in D$.
Assume further that the conformal weight of $W^\beta \otimes V^\beta$ is an integer 
for $\beta \in D$ and that 
the direct sum $U = \bigoplus_{\beta \in D} W^\beta \otimes V^\beta$ 
has a structure of a simple vertex operator algebra which extends the 
$W \otimes V$-module structure on $U$. 
The vertex operator algebra $U$ is simple, self-dual, rational, $C_2$-cofinite, and of CFT type.  

Such a triple $(U, V, W)$ appears when we consider the commutant 
of a subalgebra in a vertex operator algebra. 
In fact, the commutant of $V$ in $U$ is $W$, and the commutant of $W$ in $U$ is $V$ 
in our setting. 

One of the important examples is the parafermion vertex operator algebra 
$K(\mathfrak{g},k)$ associated with a finite dimensional simple Lie algebra $\mathfrak{g}$ 
and a positive integer $k$. 
Let $L_{\widehat{\mathfrak{g}}}(k,0)$ be a simple affine vertex operator algebra 
for the affine Kac-Moody Lie algebra $\widehat{\mathfrak{g}}$ at level $k$, 
and 
let $Q_L$ be the sublattice of the root lattice of $\mathfrak{g}$ spanned by the long roots. 
Then $L_{\widehat{\mathfrak{g}}}(k,0)$ contains a lattice vertex operator algebra 
$V_{\sqrt{k}Q_L}$,  
and 
$K(\mathfrak{g},k)$ is the commutant of $V_{\sqrt{k}Q_L}$ 
in $L_{\widehat{\mathfrak{g}}}(k,0)$.  
Moreover, 
$L_{\widehat{\mathfrak{g}}}(k,0)$ is a simple current extension 
of $K(\mathfrak{g},k) \otimes V_{\sqrt{k}Q_L}$. 
Thus $W = K(\mathfrak{g},k)$, $V = V_{\sqrt{k}Q_L}$, and $U = L_{\widehat{\mathfrak{g}}}(k,0)$ 
satisfy the above conditions. 
As to the properties of parafermion vertex operator algebras, 
see for example 
\cite{ADJR2017, ALY2014, ALY2019, DR2017}. 
Another important example of such a triple $(U,V,W)$ is known in a certain W-algebras,  
see for example \cite{ACL2017, CL2017}. 

The representation theory of simple current extensions was developed 
in \cite{Yamauchi2004}. 
Since any irreducible $U$-module is a direct sum of inequivalent irreducible 
$W \otimes V$-modules by our assumption on the vertex operator algebras $W$ and $V$, 
the argument for the representations of $U$ is much simpler 
than that for a general one.

In this paper,  
we classify the irreducible $U$-modules (Theorem \ref{thm:all_modules}), 
and establish the relationship between 
the fusion rules for $U$ and $W$ (Theorem \ref{thm:fusion_products_4_U_W}) 
under the above setting. 
The relationship between the fusion algebras of $U$ and $W$ 
is considered as well, see Section \ref{subsec:RU_RW}.  
The argument for the classification of irreducible $U$-modules is standard 
\cite{Yamauchi2004}.
Our treatment of the relationship between the fusion rules for $U$ and $W$ 
is slightly different from the proofs of 
\cite[Theorem 5.2]{ADJR2017} and \cite[Theorem 4.2]{DW2016}. 
In fact, the quantum dimension in the sense of \cite{DJX2013} plays a role in those papers, 
whereas we use a braided tensor functor \cite[Theorem 2.67]{CKM2017}, 
see Theorem \ref{thm:functor_F} below. 

One of the features of a $D$-graded simple current extension is that 
an irreducible twisted module associated with an automorphism 
induced by an element of the character group $D^\ast$ of $D$ is constructed explicitly 
\cite[Theorem 3.2]{Yamauchi2004}. 
We are concerned with not only irreducible $U$-modules but also those 
irreducible twisted $U$-modules (Theorem \ref{thm:all_modules}). 

In Section \ref{subsec:RU_RW}, we discuss a duality between the fusion algebras 
of $U$ and $W$. 
Since $\Irr(V) = \{ V^\alpha \mid \alpha \in C \}$ is a $C$-graded set of simple 
current $V$-modules, the group $C$ has a structure of a quadratic space with 
a quadratic form $q_V$ and its associated bilinear form $b_V$ 
\cite[Theorem 3.4]{vEMS2017}, see Proposition \ref{prop:bV-bilinear-rev} below. 
The bilinear form $b_V$ is non-degenerate \cite[Proposition 3.5]{vEMS2017}. 
It turns out that $U$ has a $D^\perp$-graded set $\{ U^\gamma \mid \gamma \in D^\perp \}$ 
of simple current $U$-modules, where $D^\perp = \{ \alpha \in C \mid b_V(\alpha,D)=0\}$. 
We consider an action of $D$ on $\Irr(W)$ (resp. $D^\perp$ on $\Irr(U)$) defined by 
$X \mapsto W^\beta \boxtimes_W X$ for $\beta \in D$, $X \in \Irr(W)$ 
(resp. $X \mapsto U^\gamma \boxtimes_U X$ for $\gamma \in D^\perp$, $X \in \Irr(U)$). 
We show that the set of $D$-orbits in $\Irr(W)$ and the set of $D^\perp$-orbits in $\Irr(U)$ are 
in one to one correspondence. 
The fusion algebras of $W$ and $U$ are related to each other through the correspondence. 

In the case $W = K(\mathfrak{g},k)$, $V = V_{\sqrt{k}Q_L}$, and 
$U = L_{\widehat{\mathfrak{g}}}(k,0)$,  
Corollary \ref{cor:count} and 
Theorem \ref{thm:fusion_products_4_U_W} 
correspond to \cite[Theorem 5.1]{ADJR2017} and \cite[Theorem 5.2]{ADJR2017}, 
respectively. 
Furthermore, Remark \ref{rem:underSC} is related to 
\cite[Theorem 4.7]{DR2017}.

This paper is organized as follows. 
Section \ref{sec:preliminaries} is devoted to preliminaries. 
We recall basic properties of fusion rules for vertex operator algebras. 
We also review simple current extensions of vertex operator algebras. 
Moreover, we introduce a chain of simple current extensions of special type.
In Section \ref{sec:general}, 
we study a simple current extension $U$ of a tensor product $W \otimes V$ of 
vertex operator algebras $W$ and $V$ in the above-mentioned setting. 
We discuss irreducible $U$-modules as well as irreducible twisted $U$-modules. 
We establish the relationship between 
the fusion rules for $U$ and $W$ and discuss the fusion algebras 
of $U$ and $W$. 

Our notations for vertex operator algebras and their modules are standard 
\cite{FHL1993, LL2004}.
We write $\boxtimes_V$ for the $P(z)$-tensor product $\boxtimes_{P(z)}$ of \cite{HL1994} 
for a vertex operator algebra $V$ with $z = 1$, and call it the fusion product. 
We use $\otimes$ to denote the tensor product of vertex operator algebras 
and their modules as in \cite{FHL1993}.

\section{Preliminaries}\label{sec:preliminaries}

In this section, we collect some basic properties of fusion rules and 
simple currents of vertex operator algebras. 
Moreover, we review simple current extensions of vertex operator algebras 
and their representations, including irreducible twisted modules. 
We consider a chain of simple current extensions as well. 
We assume that 
the vertex operator algebras discussed in this paper  
satisfy the following hypothesis,  unless otherwise specified.

\begin{hypothesis}\label{hypo_on_VOA}
$V$ is a simple, self-dual, rational, and $C_2$-cofinite vertex operator algebra of 
CFT-type.
\end{hypothesis}

\subsection{Fusion rules}\label{subsec:fusion_rule}

Let $V$ be a vertex operator algebra satisfying 
Hypothesis \ref{hypo_on_VOA}.  
Then for $V$-modules $M^i$, $i = 1,2,3$, 
the set $I_V \binom{M^3}{M^1 \ M^2}$ of intertwining operators of type 
$\binom{M^3}{M^1 \ M^2}$ is a finite dimensional vector space 
\cite[Corollary 5.10]{ABD2004}. 
The dimension $N^{M_3}_{M^1, M^2} = N(V)^{M_3}_{M^1, M^2}$ 
of $I_V \binom{M^3}{M^1 \ M^2}$ 
is called the fusion rule. 
A fusion product $M \boxtimes_V N$ over $V$ of any $V$-modules 
$M$ and $N$ exists \cite{HL1995c, Li1998}. 
The fusion product is commutative and associative 
\cite[Theorem 3.7]{Huang2005}. 
Denote by $\Irr(V)$ the set of equivalence classes of irreducible $V$-modules.
Then 
\begin{equation}\label{eq:def_fusion_product}
  M^1 \boxtimes_V M^2 = \sum_{M^3 \in \Irr(V)} N^{M_3}_{M^1, M^2} M^3
\end{equation}
for $M^1, M^2 \in \Irr(V)$. 
A vector space
$R(V)$ with basis $\Irr(V)$ equipped with multiplication $\boxtimes_V$ 
is called the fusion algebra of $V$. 

An irreducible $V$-module $A$ is called a simple current 
if $A \boxtimes_V X$ 
is an irreducible $V$-module for any $X \in \Irr(V)$. 
We denote by $\SC{V}$ the set of equivalence classes of 
simple current $V$-modules. 
A set $\{ A^\alpha \mid \alpha \in D\}$ of simple current $V$-modules indexed by 
an abelian group $D$ is said to be $D$-graded if 
$A^\alpha$, $\alpha \in D$, are inequivalent to each other 
with $A^0 = V$ and
$A^\alpha \boxtimes_V A^\beta = A^{\alpha + \beta}$ for $\alpha, \beta \in D$. 
It was shown in \cite[Corollary 1]{LY2008} that $\SC{V}$ 
is graded by a finite abelian group. 
The inverse of $A \in \SC{V}$ with respect to the fusion product is 
its contragredient module $A'$.

\begin{lemma}\label{lem:Atensor}
  Let $V$ be a vertex operator algebra satisfying 
  Hypothesis \ref{hypo_on_VOA}. 
  Then 
  \begin{equation*}
    \dim I_V \binom{A \boxtimes_V M^3}{A \boxtimes_V M^1 \quad M^2}
    = \dim I_V \binom{A \boxtimes_V M^3}{M^1 \quad A \boxtimes_V M^2}
    = \dim I_V \binom{M^3}{M^1 \quad M^2}
  \end{equation*}
  for $A \in \SC{V}$ and $M^1$, $M^2$, $M^3 \in \Irr(V)$.
\end{lemma}

\begin{proof}
Since the fusion product is commutative and associative, 
we have
\begin{equation*}
  (A \boxtimes_V M^1) \boxtimes_V M^2 
  = M^1 \boxtimes_V (A \boxtimes_V M^2)
  = \sum_{M^3 \in \Irr(V)} N^{M_3}_{M^1, M^2} A \boxtimes_V M^3 
\end{equation*}
by \eqref{eq:def_fusion_product}. Thus the assertion holds. 
\end{proof}

The following proposition 
will be used later. 

\begin{proposition} $($\cite[Theorem 2.10]{ADL2005}$)$ 
\label{prop:tensor_prod_fusion}
Let $V^1$ and $V^2$ be rational and $C_2$-cofinite vertex operator algebras,
and let $M^i$ and $N^i$, $i = 1,2,3$, be modules for 
$V^1$ and $V^2$, respectively. 
Then
\begin{equation*}
  I_{V^1 \otimes V^2} \binom{M^3 \otimes N^3}{M^1  \otimes N^1 \quad M^2 \otimes N^2}
  \cong I_{V^1} \binom{M^3}{M^1 \ M^2} \otimes I_{V^2} \binom{N^3}{N^1 \ N^2}
\end{equation*}
as vector spaces.
\end{proposition}

\subsection{Simple currents}\label{subsec:simple_current_ext}

Let $V$ be a vertex operator algebra satisfying 
Hypothesis \ref{hypo_on_VOA}.  
We denote by $\CC_V$ the category of $V$-modules. 
It is known that $\CC_V$ is a modular tensor category \cite{Huang2008b}, 
see also \cite{Huang1995, Huang2005, Huang2008a, HL1995c} 
and  references therein. 
In fact, let 
\begin{equation}\label{eq:braiding}
  c_{X,Y} = \CC_{P(1)} : X \boxtimes_V Y \to Y \boxtimes_V X
\end{equation}
be as in \cite[Section 1]{Huang2008b}, 
and let
\begin{equation}\label{eq:twist}
  \theta_X = e^{2\pi\sqrt{-1}L(0)} : X \to X
\end{equation}
be as in \cite[Section 4]{Huang2008b}, where $L(0)$ is the Virasoro zero-mode. 
Recall also that $X'$ is the contragredient module of a $V$-module $X$. 

\begin{theorem} $($\cite[Theorem 4.5]{Huang2008b}$)$ 
\label{thm:module_category}
  Let $V$ be a vertex operator algebra satisfying 
  Hypothesis \ref{hypo_on_VOA}. 
  Then the category $\CC_V$ of $V$-modules is a modular tensor category 
  with the tensor product $\boxtimes_V$, the dual $X'$, 
  the braiding $c$, the twist $\theta$, and the unit object being the ajoint module $V$.
\end{theorem}

The twist $\theta$ is a family $\theta_X : X \to X$ of natural isomorphisms  
such that 
\begin{equation*}
  \theta_{X \boxtimes_V Y} = (\theta_X \boxtimes_V \theta_Y) 
  \circ c_{Y,X} \circ c_{X,Y}.
\end{equation*}
By the naturality of $c_{X,Y}$, we have
\begin{equation*}
  \theta_{X \boxtimes_V Y} 
  = c_{Y,X} \circ (\theta_Y \boxtimes_V \theta_X) \circ c_{X,Y} 
  = c_{Y,X} \circ c_{X,Y} \circ (\theta_X \boxtimes_V \theta_Y).
\end{equation*}

Let $M_{X,Y} = c_{Y,X} \circ c_{X,Y} : X \boxtimes_V Y \to X \boxtimes_V Y$ 
be the monodromy isomorphism.
Then 
\begin{equation}\label{eq:monodromy_property}
  \theta_{X \boxtimes_V Y} = M_{X,Y} \circ (\theta_X \boxtimes_V \theta_Y).
\end{equation}

For a $V$-module $X$, we denote its conformal weight by $h(X)$.
It was shown in \cite[Theorem 11.3]{DLM2000} that $h(X)$ 
is a rational number. 
We define $\Q/\Z$-valued maps 
on $\Irr(V)$ and $\SC{V}\times \Irr(V)$ by
\begin{equation}\label{eq:quad}
  \begin{split}
  q_V(X) &= h(X) + \Z,\\
  b_V(A,X) &= h(A \boxtimes_V X) - h(A) - h(X) + \Z
  \end{split}
\end{equation}
for $A \in \SC{V}$ and $X \in \Irr(V)$. 
Then
\begin{equation}\label{eq:bVqV}
  b_V(A,X) = q_V(A \boxtimes_V X) - q_V(A) - q_V(X).
\end{equation}
The maps $q_V$ and $b_V$ were introduced in \cite[Section 3]{vEMS2017} 
in the case where $\SC{V}=\Irr(V)$, 
see also \cite[Section 2]{M2016}.

It follows from \eqref{eq:twist} that 
\begin{equation*}
  \theta_X = e^{2\pi\sqrt{-1}h(X)} \iden_X, 
\end{equation*}
if $X \in \Irr(V)$.
Moreover, we have
\begin{equation}\label{eq:MAX}
  M_{A,X} = e^{2\pi \sqrt{-1} b_V(A,X)} \iden_{A \boxtimes_V X}
\end{equation}
for $A \in \SC{V}$ and $X \in \Irr(V)$ by \eqref{eq:monodromy_property}.

\begin{proposition}\label{prop:bV-bilinear-rev}
  Let $V$ be a vertex operator algebra satisfying 
  Hypothesis \ref{hypo_on_VOA}. 
  Then for $A$, $B\in \SC{V}$ and  $X \in \Irr(V)$, 
  the following assertions hold.

  \textup{(1)} $b_V(A\boxtimes_V B, X) = b_V(A, X)+b_V(B, X)$.

  \textup{(2)} $b_V(A, B\boxtimes_V X) = b_V(A, B)+b_V(A, X)$.

  \textup{(3)} $q_V(A^{\boxtimes n}) = n^2q_V(A)$ for $n\in \Z$.
\end{proposition}

\begin{proof}
Let 
$a_{X,Y,Z} : (X \boxtimes_V Y)  \boxtimes_V Z \to 
X \boxtimes_V (Y  \boxtimes_V Z)$ be the associativity isomorphism. 
For simplicity, we write $\boxtimes$ for the fusion product $\boxtimes_V$. 
By the hexagon axioms for the braided tensor category structure on $\CC_V$,  
we have
\begin{equation}\label{eq:hex-1}
  a_{B,X,A} \circ c_{A, B \boxtimes X} \circ a_{A,B,X} 
  = (\iden_B \boxtimes c_{A,X}) \circ a_{B,A,X} \circ (c_{A,B} \boxtimes \iden_X),
\end{equation}
\begin{equation}\label{eq:hex-2}
a_{X,A,B}^{-1} \circ c_{A \boxtimes B, X} \circ a_{A,B,X}^{-1} 
= (c_{A,X} \boxtimes \iden_B) \circ a_{A,X,B}^{-1} \circ (\iden_A \boxtimes c_{B,X}).
\end{equation}
Replace $A$, $B$, and $X$ with $X$, $A$, and $B$, repectively 
in \eqref{eq:hex-1}. 
Then 
\begin{equation}\label{eq:hex-1'}
a_{A,B,X} \circ c_{X, A \boxtimes B} \circ a_{X,A,B} 
= (\iden_A \boxtimes c_{X,B}) \circ a_{A,X,B} \circ (c_{X,A} \boxtimes \iden_B).
\end{equation}
Similarly, replace $A$, $B$, and $X$ with $B$, $X$, and $A$, respectively 
in \eqref{eq:hex-2}. 
Then 
\begin{equation}\label{eq:hex-2'}
a_{A,B,X}^{-1} \circ c_{B \boxtimes X, A} \circ a_{B,X,A}^{-1} 
= (c_{B,A} \boxtimes \iden_X) \circ a_{B,A,X}^{-1} \circ (\iden_B \boxtimes c_{X,A}).
\end{equation}

The composition of the left hand side of \eqref{eq:hex-2} and 
the left hand side of \eqref{eq:hex-1'} is 
\begin{equation*}
\begin{split}
& a_{A,B,X} \circ c_{X, A \boxtimes B} \circ a_{X,A,B} 
\circ a_{X,A,B}^{-1} \circ c_{A \boxtimes B, X} \circ a_{A,B,X}^{-1}\\
& \qquad = a_{A,B,X} \circ M_{A \boxtimes B, X} \circ a_{A,B,X}^{-1}\\
& \qquad = e^{2\pi\sqrt{-1} b_V(A \boxtimes B, X)}  \iden_{A \boxtimes (B \boxtimes X)}.
\end{split}
\end{equation*}
The composition of the right hand side of \eqref{eq:hex-2} and 
the right hand side of \eqref{eq:hex-1'} is 
\begin{equation*}
\begin{split}
& (\iden_A \boxtimes c_{X,B}) \circ a_{A,X,B} \circ (c_{X,A} \boxtimes \iden_B) 
\circ (c_{A,X} \boxtimes \iden_B) \circ a_{A,X,B}^{-1} \circ (\iden_A \boxtimes c_{B,X})\\
& \qquad = (\iden_A \boxtimes c_{X,B}) \circ a_{A,X,B} \circ (M_{A,X} \boxtimes \iden_B) 
\circ a_{A,X,B}^{-1} \circ (\iden_A \boxtimes c_{B,X})\\
& \qquad = e^{2\pi \sqrt{-1}b_V(A,X)} (\iden_A \boxtimes c_{X,B}) \circ (\iden_A \boxtimes c_{B,X})\\
& \qquad = e^{2\pi \sqrt{-1}b_V(A,X)} (\iden_A \boxtimes M_{B,X})\\
& \qquad = e^{2\pi \sqrt{-1}(b_V(A,X) + b_V(B,X))}
\iden_{A \boxtimes (B \boxtimes X)}.
\end{split}
\end{equation*}
Thus the assertion (1) holds.

The composition of the left hand side of \eqref{eq:hex-1} and 
the left hand side of \eqref{eq:hex-2'} is 
\begin{equation*}
\begin{split}
& a_{A,B,X}^{-1} \circ c_{B \boxtimes X, A} \circ a_{B,X,A}^{-1} 
\circ a_{B,X,A} \circ c_{A, B \boxtimes X} \circ a_{A,B,X}\\
& \qquad = a_{A,B,X}^{-1} \circ M_{A, B \boxtimes X} \circ a_{A,B,X}\\
& \qquad = e^{2\pi\sqrt{-1} b_V(A, B \boxtimes X)} \iden_{(A \boxtimes B) \boxtimes X}.
\end{split}
\end{equation*}
The composition of the right hand side of \eqref{eq:hex-1} and 
the right hand side of \eqref{eq:hex-2'} is 
\begin{equation*}
\begin{split}
& (c_{B,A} \boxtimes \iden_X) \circ a_{B,A,X}^{-1} \circ (\iden_B \boxtimes c_{X,A}) 
\circ (\iden_B \boxtimes c_{A,X}) \circ a_{B,A,X} \circ (c_{A,B} \boxtimes \iden_X)\\
& \qquad = (c_{B,A} \boxtimes \iden_X) \circ a_{B,A,X}^{-1} \circ (\iden_B \boxtimes M_{A,X}) 
\circ a_{B,A,X} \circ (c_{A,B} \boxtimes \iden_X)\\
& \qquad = e^{2\pi\sqrt{-1} b_V(A, X)} (c_{B,A} \boxtimes \iden_X) \circ (c_{A,B} \boxtimes \iden_X)\\
& \qquad = e^{2\pi\sqrt{-1} b_V(A, X)} (M_{A,B} \boxtimes \iden_X)\\
& \qquad = e^{2\pi\sqrt{-1} (b_V(A, B) + b_V(A,X))} \iden_{(A \boxtimes B) \boxtimes X}.
\end{split}
\end{equation*}
Thus the assertion (2) holds.

The above results imply that the restriction of $b_V$ to 
$\SC{V}\times \SC{V}$ is a symmetric $\Z$-bilinear form. 
Since $q_V(V) = 0$, and since 
$q_V(A^{\boxtimes (-n)}) = q_V(A^{\boxtimes n})$ as 
the contragredient module of $A^{\boxtimes n}$ is $A^{\boxtimes (-n)}$, 
the assertion (3) holds for all $n \in \Z$ by \eqref{eq:bVqV}.
\end{proof}

The assertion (3) of the above proposition implies the following corollary, 
see \cite[Theorem 3.4]{vEMS2017}.

\begin{corollary}\label{cor:quad}
  $\SC{V}$ carries a structure of a finite quadratic space with 
  quadratic form $q_V$ and associated bilinear form $b_V$.
\end{corollary}

We also have the next proposition.

\begin{proposition}\label{prop:nondegeneracy-rev}
  Let $V$ be a vertex operator algebra satisfying 
  Hypothesis \ref{hypo_on_VOA}. 
  If $A \in \SC{V}$ satisfies the condition that $b_V(A,X) = b_V(V,X)$ 
  for all $X \in \Irr(V)$, then $A = V$.
\end{proposition}

\begin{proof}
Consider the categorical trace $\tr_{\CC_V}(M_{A,X})$ of $M_{A,X}$ 
associated with the twist $\theta$. 
We have
\[
  \tr_{\CC_V}(M_{A,X}) = e^{2\pi \sqrt{-1} b_V(A,X)} 
  \tr_{\CC_V}(\iden_A) \tr_{\CC_V}(\iden_X)
\]
by \eqref{eq:MAX}. 
Since the category $\CC_V$ of $V$-modules is a modular tensor category 
by Theorem \ref{thm:module_category}, 
the $S$-matrix $S = (s_{XY})_{X,Y \in \Irr(V)}$ with $s_{XY} = \tr_{\CC_V}(M_{X,Y})$ 
is non-degenerate. 
If $b_V(A,X) = b_V(V,X)$ for all $X \in \Irr(V)$, 
then the $A$-th row of $S$ is a constant multiple of the $V$-th row of $S$.  
Hence $A = V$.
\end{proof}

In the case $\SC{V} = \Irr(V)$, we have the following corollary, 
see \cite[Proposition 3.5]{vEMS2017}.

\begin{corollary}\label{cor:bilinear}
  The map $b_V$ is a non-degenerate symmetric $\Z$-bilinear form on $\Irr(V)$ 
  if $\SC{V} = \Irr(V)$.
\end{corollary}

Suppose $\SC{V}=\{ V^\alpha \mid \alpha \in C\}$ is graded 
by a finite abelian group $C$ with $V=V^0$ 
and $V^\alpha\boxtimes_V V^\beta = V^{\alpha+\beta}$ 
for $\alpha$, $\beta\in C$. 
We regard $q_V$ as a quadratic form on $C$ by setting 
$q_V(\alpha)=q_V(V^\alpha)$. 
Likewise, we set $b_V(\alpha,\beta) = b_V(V^\alpha,V^\beta)$. 
For any subgroup $D$ of $C$, we obtain a $D$-graded 
set $\{ V^\alpha \mid\alpha \in D \}$ 
of simple current $V$-modules.
Set
\begin{equation*}
  D^\perp = \{ \alpha \in C \mid b_V(\alpha,D)=0\}.
\end{equation*}

Then $D^\perp$ is a subgroup of $C$.
We say that $\alpha \in C$ is isotropic if $q_V(\alpha)=0$,  
and that $D$ is totally isotropic if $q_V(D)=0$. 
Note that  $D\subset D^\perp$ if $D$ is totally isotropic, 
but the converse is not generally true 
as $b_V(\alpha,\alpha)=2q_V(\alpha)$. 
If $\SC{V} = \Irr(V)$, then we have $\abs{C} = \abs{D} \abs{D^\perp}$ 
as $b_V$ is non-degenerate by Corollary \ref{cor:bilinear}.

Let $D$ be a totally isotropic subgroup of $C$. 
Then the direct sum
\[
  V_D = \bigoplus_{\alpha \in D}V^\alpha
\]
is closed under the fusion product, 
and $V_D$ is of integral weight. 
The algebraic structure of $V_D$ has been studied for many years, 
see for example \cite{Carnahan2014, CKL2015, CKM2017, vEMS2017, Yamauchi2004} 
and references therein. 
Here we cite \cite[Theorem 3.12]{CKL2015}, 
see also \cite[Theorem 3.9]{CKL2015}.

\begin{theorem} $($\cite[Theorem 3.12]{CKL2015}$)$ \label{thm:VOSA}
  Let $V$ be a vertex operator algebra satisfying Hypothesis \ref{hypo_on_VOA}. 
  Suppose $\SC{V}=\{ V^\alpha \mid \alpha \in C\}$ is graded 
  by a finite abelian group $C$. 
  Let $V_D = \bigoplus_{\alpha \in D}V^\alpha$ 
  for a totally isotropic subgroup $D$ of $C$. 
  Then the direct sum $V_D = \bigoplus_{\alpha \in D}V^\alpha$ has either
  a simple vertex operator algebra structure 
  or a simple vertex operator superalgebra structure, 
  which extends the $V$-module structure on $V_D$.
\end{theorem}

\begin{remark}\label{rmk:with_positivity}
  Let $V$ be a vertex operator algebra satisfying Hypothesis \ref{hypo_on_VOA}.  
  Assume further that the conformal weight of any irreducible $V$-module 
  is positive except for the adjoint module $V$. 
  Then for any irreducible $V$-module $X$, the categorical dimension 
  $\dim_{\CC_V} X = \tr_{\CC_V}(\iden_X)$ of $X$ 
  coincides with the quantum dimension 
  $\qdim X$ of $X$ in the sense of \cite[Definition 3.1]{DJX2013} by 
  \cite[Eq. (4.1)]{DJX2013}, see also \cite[Proposition 3.11]{DLN2015}. 
  This in particular implies that $\dim_{\CC_V} X \ge 1$. 
  In this case, the direct sum $V_D = \bigoplus_{\alpha \in D}V^\alpha$ 
  has a vertex operator algebra structure, 
  which extends the $V$-module structure on $V_D$, 
  see also \cite[Theorem 3.2.12]{Carnahan2014} and \cite[Theorem 4.2]{vEMS2017}.
\end{remark}

If the direct sum $V_D = \bigoplus_{\alpha \in D}V^\alpha$ 
associated with a totally isotropic subgroup $D$ of $C$ has a structure of 
a vertex operator algebra which extends the $V$-module structure 
on $V_D$, then such a vertex operator algebra structure on $V_D$ is unique 
\cite[Proposition 5.3]{DM2004}. 
The vertex operator algebra $V_D$ is called 
a $D$-graded simple current extension of $V=V^0$. 
It is known that $V_D$ also satisfies Hypothesis \ref{hypo_on_VOA}, 
see \cite[Theorem 2.14]{Yamauchi2004}. 
Indeed, since the contragredient $V$-module $(V^\alpha)'$ of $V^\alpha$ is 
$V^{-\alpha}$, the contragredient $V_D$-module of the adjoint 
module $V_D$ is $\bigoplus_{\alpha \in D}V^{-\alpha} = V_D$ 
by the uniqueness of simple current extension. 
Thus $V_D$ is self-dual. 

The character group $D^* = \Hom(D,\C^\times)$ of $D$ 
naturally acts on $V_D = \bigoplus_{\alpha \in D} V^\alpha$.
In fact, for $\chi \in D^*$ a scalar multiplication by 
$\chi(\alpha)$ on $V^\alpha$, $\alpha \in D$, 
is an automorphism of the vertex oerator algebra $V_D$. 
Thus we can regard $D^*$ as a subgroup of $\Aut(V_D)$. 
The fixed point subalgebra of $V_D$ by $D^\ast$ is 
$(V_D)^{D^\ast} = V$.
If an automorphism $g$ of $V_D$ acts trivially on $V$, 
then $g$ agrees with some $\chi \in D^\ast$. 

The group $D$ acts on the set $\Irr(V)$ by $X \mapsto V^\alpha \boxtimes_V X$
for $\alpha \in D$ and $X \in \Irr(V)$.
Let $\mathcal{O} \subset \Irr(V)$ be a $D$-orbit, 
and take $X \in \mathcal{O}$.
It follows from (1) of Proposition \ref{prop:bV-bilinear-rev} that the map
\begin{equation*}
  \xi_X : D \to \Q/\Z; \quad \alpha \mapsto b_V(V^\alpha,X)
\end{equation*}
is a group homomorphism.
Since $D\subset D^\perp$, it follows from (2) of 
Proposition \ref{prop:bV-bilinear-rev} that 
this map is independent of the choice of $X \in \mathcal{O}$, 
and we can denote $\xi_X$ by $\xi_{\mathcal{O}}$.

By exponentiation we obtain a linear character 
\begin{equation*}
  \widehat{\xi}_X(\alpha) 
  = \exp(2\pi\sqrt{-1}\, \xi_X(\alpha)) \in \C^\times
\end{equation*}
of $D$.
Therefore, each $X \in \Irr(V)$, 
as well as a $D$-orbit $\mathcal{O}$ in $\Irr(V)$, 
defines an automorphism $\widehat{\xi}_X = \widehat{\xi}_{\mathcal{O}}$ 
of the vertex operator algebra $V_D$.
If $D$ acts on $\mathcal{O}$ freely, it is known that 
the $D$-orbit $\mathcal{O}$ 
uniquely defines an irreducible twisted $V_D$-module by 
\cite[Theorem 3.3]{Yamauchi2004}. 
In fact, the following theorem holds.

\begin{theorem}\label{thm:induced_module} 
  Let $V$ be a vertex operator algebra satisfying Hypothesis \ref{hypo_on_VOA},  
  and let
  $\SC{V} = \{ V^\alpha \mid \alpha \in C\}$ be the set of equivalence classes 
  of simple current $V$-modules graded by $C$.
  Let $V_D$ be a simple current extension of $V$ associated with 
  a totally isotropic subgroup $D$ of $C$.
  Let $X \in \Irr(V)$, and 
  let $\mathcal{O } = \{ V^\alpha \boxtimes_V X \mid \alpha \in D\}$ be its 
  $D$-orbit in $\Irr(V)$.
  Suppose $D$ acts on $\mathcal{O}$ freely.
  Then there exists a unique structure of an irreducible 
  $\widehat{\xi}_X$-twisted $V_D$-module on the direct sum 
  $V_D\boxtimes_V X = \bigoplus_{\alpha \in D} V^\alpha \boxtimes_V X$ 
  of irreducible $V$-modules which extends the $V$-module structure on $X$.
\end{theorem}

As to the notion of a $g$-twisted module for a vertex operator algebra 
with respect to its automorphism $g$, 
we adopt the definition in \cite{DLM2000}. 
Thus a $g$-twisted module in \cite{Yamauchi2004} means 
a $g^{-1}$-twisted  module in this paper.

The theory of vertex operator algebra extensions has been developed in more general 
contexts than here, see for exmple \cite{CKL2015, CKM2017, HKL2015, KO2002}. 
We apply the general theory of vertex operator algebra extensions to our case as follows: 
$V$ is a vertex operator algebra satisfying Hypothesis \ref{hypo_on_VOA},  
$\CC = \CC_V$ is the category of $V$-modules, 
$A = V_e$ is the simple current extension $V_D$ of $V$ 
associated with a totally isotropic subgroup $D$ of $C$, 
and $\Rep^0 V_D$ is the category of $V_D$-modules, 
see \cite[Section 3]{CKL2015}, \cite[Section 3]{CKM2017}, \cite[Section 3]{HKL2015}, 
and \cite[Section 5]{KO2002}. 

Indeed, the simple current extension $V_D$ of $V$  
is a commutative associative algebra $A = V_e$ defined as in \cite[Definition 1.1]{KO2002} 
(see also \cite[Definition 2.2]{CKM2017}, \cite[Definition 3.1]{HKL2015})
in the braided tensor category $\CC_V$ of $V$-modules 
\cite[Theorem 3.2, Remark 3.3]{HKL2015}. 
A category $\Rep A = \Rep V_D$ is introduced in \cite[Definition 1.2]{KO2002},
which a $\C$-linear abelian monoidal category 
\cite[Lemma 1.4, Theorem 1.5]{KO2002}, 
see also \cite[Theorems 2.9 and 2.53 ]{CKM2017}. 
A functor
\[
  F : \CC_V \to \Rep V_D; \quad X \mapsto V_D \boxtimes_V X
\]
is defined, and it is shown that $F$ is a tensor functor \cite[Theorem 1.6]{KO2002}, 
see also \cite[Theorem 2.59]{CKM2017}. 
A full subcategory $\Rep^0 A$ of $\Rep A$ is introduced in \cite[Definition 1.8]{KO2002}. 
Its properties in the general theory are studied in \cite{CKM2017, KO2002}. 
In our case $\Rep^0 A = \Rep^0 V_D$ is the braided tensor category of 
$V_D$-modules in $\CC_V$ \cite[Theorem 3.4]{HKL2015}.

As in \cite[Definition 2.66]{CKM2017}, 
we denote by $\CC_V^0$ the full subcategory of $\CC_V$ consisting of 
the objects $X$ of $\CC_V$ such that $F(X)$ is an object of $\Rep^0 V_D$. 
A necessary and sufficient condition on an object $X$ of $\CC_V$ for which $F(X)$ 
is an object of $\Rep^0 V_D$ can be found in \cite[Proposition 2.65]{CKM2017}. 
We will use \cite[Theorem 2.67]{CKM2017} in the following form later.

\begin{theorem} $($\cite[Theorem 2.67]{CKM2017}$)$ \label{thm:functor_F}
  The category $\CC_V^0$ is a $\C$-linear additive braided monoidal category 
  with structures induced from $\CC_V$, 
  and the restriction $F : \CC_V^0 \to \Rep^0 V_D$ of the functor $F$ to $\CC_V^0$ is 
  a braided tensor functor.
\end{theorem}

A stronger assertion in a more general context can be found in 
\cite[Theorem 3.68]{CKM2017}.

\subsection{A chain of simple current extensions}\label{subsec:sequence_of_SCE}

Let $k, m \in \Z_{> 0}$, and let $\Z \gamma$ be a rank one lattice spanned by 
an element $\gamma$ with square norm $\la \gamma,\gamma \ra = 2km$. 
Let $W$ be a vertex operator algebra satisfying Hypothesis \ref{hypo_on_VOA}, 
and let $\{ W^j \mid j \in \Z_k \}$ be a $\Z_k$-graded set of simple current 
$W$-modules with $W^0 = W$. 
We also recall the fusion rule for $V_{\Z \gamma}$ \cite[Chapter 12]{DL1993}. 
Assume that $U = \bigoplus_{j = 0}^{k-1} W^j \otimes V_{\Z \gamma - j\gamma/k}$ is  
a vertex operator algebra which is a $\Z_k$-graded simple current extension of 
$W \otimes V_{\Z \gamma}$. 

Let $s \in \Z$ with $m+sk > 0$,  
and let $\gamma'$ be an element with square norm 
$\la \gamma', \gamma' \ra = 2k(m+sk)$. 
Then both the conformal weight of $V_{\Z \gamma - j\gamma/k}$ and 
the conformal weight of 
$V_{\Z \gamma' - j\gamma'/k}$ are congruent to $j^2m/k$ modulo $\Z$. 
Thus $W^j \otimes V_{\Z \gamma' - j\gamma'/k}$, $0 \le j < k$, are of 
integral weight, 
and the direct sum 
$\widetilde{U} = \bigoplus_{j = 0}^{k-1} W^j \otimes V_{\Z \gamma' - j\gamma'/k}$ 
has a unique vertex operator algebra structure as an extension of 
$W \otimes V_{\Z \gamma'}$. 
Hence the following theorem holds. 

\begin{theorem}\label{thm:deform_VOA}
Let $k, m \in \Z_{> 0}$, and let $\gamma$ be an element with $\la \gamma,\gamma \ra = 2km$. 
Let $W$ be a vertex operator algebra satisfying Hypothesis \ref{hypo_on_VOA}, 
and let $\{ W^j \mid j \in \Z_k \}$ be a $\Z_k$-graded set of simple current 
$W$-modules with $W^0 = W$. 
If $U = \bigoplus_{j = 0}^{k-1} W^j \otimes V_{\Z \gamma - j\gamma/k}$ 
is a $\Z_k$-graded simple current extension of $W \otimes V_{\Z \gamma}$, 
then for $s \in \Z$ with $m+sk > 0$ and an element $\gamma'$ with 
$\la \gamma', \gamma' \ra = 2k(m+sk)$, 
we have a $\Z_k$-graded simple current extension 
$\widetilde{U} = \bigoplus_{j = 0}^{k-1} W^j \otimes V_{\Z \gamma' - j\gamma'/k}$ 
of $W \otimes V_{\Z \gamma'}$. 
The vertex operator algebra $\widetilde{U}$ satisfies Hypothesis \ref{hypo_on_VOA}.
\end{theorem}

An important example is that $W$ is the parafermion vertex operator algebra 
$K(\mathfrak{sl}_2, k)$ with $m = 1$ and 
$U = L_{\widehat{\mathfrak{sl}}_2}(k,0)$ is the simple affine vertex operator algebra 
at level $k$. 
In this case, $\widetilde{U}$ was studied in \cite{Adamovic2007} and 
\cite[Section 4.4]{CKM2017}.

\section{Irreducible $U$-modules and fusion rules for $W$ and $U$}\label{sec:general}

In this section, we assume the following hypothesis.

\begin{hypothesis}\label{hypo:W_V}
  \textup{(1)} $W$ and $V$ are vertex operator algebras satisfying  
  Hypothesis \ref{hypo_on_VOA}.

  \textup{(2)} All the irreducible $V$-modules are simple currents, and 
  $\Irr(V) = \{ V^\alpha \mid \alpha \in C\}$ is $C$-graded by a finite abelian group $C$ 
  with $V^0 = V$ and $V^\alpha \boxtimes_V V^\beta = V^{\alpha+\beta}$ 
  for $\alpha$, $\beta \in C$.  

  \textup{(3)} $D$ is a subgroup of $C$, and $\{ W^\alpha \mid \alpha \in D\}$ is 
  a $D$-graded set of simple current $W$-modules with
  $W^0 = W$ and $W^\alpha \boxtimes_V W^\beta = W^{\alpha+\beta}$ 
  for $\alpha$, $\beta \in D$.  

  \textup{(4)} The direct sum
  \begin{equation*}
    U = \bigoplus_{\beta \in D} W^\beta \otimes V^\beta
  \end{equation*}
  has a structure of a simple vertex operator algebra which extends the 
  $W \otimes V$-module structure on $U$. 
\end{hypothesis}

We discuss representations of $U$, and establish the relationship between 
the fusion rules for $U$ and $W$ under Hypothesis \ref{hypo:W_V}. 
Note that the tensor product $W \otimes V$ also satisfies Hypothesis \ref{hypo_on_VOA}.
Condition (2) of Hypothesis \ref{hypo:W_V} implies that 
$b_V$ is a non-degenerate symmetric $\Z$-bilinear form on $C$ by 
Corollary \ref{cor:bilinear} with 
$b_V(\alpha,\beta) = b_V(V^\alpha,V^\beta)$ for $\alpha$, $\beta\in C$.
Let
\begin{equation*}
  D^\perp = \{ \alpha\in C \mid b_V(\alpha,D)=0\}.
\end{equation*}
Then $\abs{C} = \abs{D} \abs{D^\perp}$ as $b_V$ is non-degenerate. 

The following lemma is a direct consequence of Hypothesis \ref{hypo:W_V}.

\begin{lemma}\label{lem:double_commutant+}
  The commutant of $W$ in $U$ is $V$, and the commutant of $V$ in $U$ is $W$.
\end{lemma}

\subsection{$\Irr(W)$ and $\Irr(U)$}

The group $D$ acts on the set $\Irr(W)$ by 
\begin{equation}\label{eq:Daction+}
  X\mapsto W^\beta \boxtimes_W X
\end{equation}
for $\beta \in D$ and $X\in \Irr(W)$.
Let 
\begin{equation}\label{eq:OW}
  \Irr(W) = \bigcup_{i \in I} \mathcal{O}^i_W
\end{equation}
be the $D$-orbit decomposition.
For $i\in I$, let 
\begin{equation*}
  D_i = \{ \beta \in D \mid W^\beta \boxtimes_W X \cong X \mbox{ for } 
  X \in \mathcal{O}^i_W\} .
\end{equation*}

Since $D$ is abelian, $D_i$ is the stabilizer of $X$ for any 
$X \in \mathcal{O}^i_W$.
Therefore, the length of the orbit $\mathcal{O}^i_W$ is $[D:D_i]$.

\begin{lemma}\label{lem:Di_D}
  $D_i\subset D^\perp$ for $i\in I$.
\end{lemma}

\begin{proof}
Let $\gamma\in D_i$, $\beta\in D$, and $X \in \mathcal{O}_W^i$.
Then by (2) of Proposition \ref{prop:bV-bilinear-rev}, 
\[
  b_W(W^\beta,X)
  = b_W(W^\beta,W^\gamma \boxtimes_W X)
  = b_W(W^\beta,W^\gamma) + b_W(W^\beta,X), 
\]
so $b_W(W^\beta,W^\gamma)=0$. 
Since $q_{W \otimes V}(W^\beta \otimes V^\beta) = 0$ for $\beta \in D$ 
by Hypothesis \ref{hypo:W_V}, we have $q_W(W^\beta) = - q_V(V^\beta)$. 
Thus $b_W(W^\beta,W^\gamma) = -b_V(\beta,\gamma)$. 
Hence $\gamma \in D^\perp$.
\end{proof}

For each $i \in I$, we pick $W^{i,0} \in \mathcal{O}^i_W$ and fix it.
Here we assume that $0 \in I$, and assign $W^{0,0}=W$ so that 
$\mathcal{O}^0_W=\{ W^\beta \mid \beta \in D\}$. 
We set
\begin{equation}\label{eq:def_Wib+}
  W^{i,\beta} = W^\beta \boxtimes_V W^{i,0}
\end{equation}
for $i\in I$ and $\beta\in D$.
Then $\mathcal{O}^i_W=\{ W^{i,\beta} \mid \beta\in D\}$. 
The next lemma holds.

\begin{lemma}\label{lem:irr_W}
  Let $i$, $i' \in I$ and $\beta$, $\beta'\in D$.
  Then $W^{i,\beta} \cong W^{i',\beta'}$ as $W$-modules if and only if $i=i'$ and 
  $\beta \equiv \beta' \pmod{D_i}$.
\end{lemma}

For $i$, $i' \in I$ and $\beta$, $\beta' \in D$, we write
\begin{equation}\label{eq:sim}
  (i,\beta) \sim (i',\beta') 
\end{equation}
if $W^{i,\beta} \cong W^{i',\beta'}$ as $W$-modules.
This is an equivalence relation on $I\times D$,   
and every element of $\Irr(W)$ is uniquely indexed by an element of $I\times D/{\sim}$ 
as $W^{i,\beta}$.
\begin{equation}\label{eq:irr_W}
  \Irr(W) = \{ W^{i,\beta} \mid (i,\beta) \in (I\times D)/{\sim}\}.
\end{equation}
Here and further we identify $(i,\beta)\in I\times D$ with its equivalence class in 
$(I\times D)/{\sim}$ by abuse of notation.

We define an action of $D$ on $\Irr(W\otimes V)$ by 
\begin{equation}\label{eq:diag_action}
  X\mapsto (W^\beta \otimes V^\beta)\boxtimes_{W\otimes V} X
\end{equation}
for $\beta \in D$ and $X\in \Irr(W\otimes V)$.

\begin{lemma}\label{lem:free}
  $D$ acts on $\Irr(W\otimes V)$ freely by \eqref{eq:diag_action}.
\end{lemma}

\begin{proof}
Let $\beta \in D$ and $X \in \Irr(W \otimes V)$.
By \eqref{eq:irr_W}, $X$ is isomorphic to
$W^{i,\gamma} \otimes V^\alpha$ for some $(i,\gamma) \in I\times D$ and $\alpha \in C$.
Then 
\[
  (W^\beta\otimes V^\beta)\boxtimes_{W\otimes V} X
  \cong 
  (W^\beta \boxtimes_W W^{i,\gamma}) \otimes (V^\beta \boxtimes_V V^\alpha)
  \cong W^{i,\gamma+\beta} \otimes V^{\alpha+\beta}.
\]
Since $V^\alpha \cong V^{\alpha+\beta}$ as $V$-modules if and only if $\beta=0$, 
the group action in \eqref{eq:diag_action} is free.
\end{proof}

By Theorem \ref{thm:induced_module}, every irreducible $W\otimes V$-module can be 
uniquely extended to an irreducible $\chi$-twisted $U$-module for some $\chi\in D^*$.
To describe irreducible twisted $U$-modules precisely, 
we introduce two $\Q/\Z$-valued maps on $D$ by 
\begin{equation}\label{eq:two_maps}
\begin{split}
  \eta_\alpha(\beta)&= b_V(V^\beta,V^\alpha),\\
  \xi_{i,\alpha} (\beta)&= b_{W \otimes V}(W^\beta \otimes V^\beta, W^{i,0} \otimes V^\alpha) 
  = b_W(W^\beta,W^{i,0}) + \eta_\alpha(\beta)
\end{split}
\end{equation}
for $i \in I$, $\alpha \in C$, and $\beta \in D$, where $b_W$ and $b_V$ are defined 
as in \eqref{eq:quad}.
It follows from Proposition \ref{prop:bV-bilinear-rev} that $\eta_\alpha$ and $\xi_{i,\alpha}$ are 
$\Z$-linear maps on $D$. 
Thus we can define linear characters $\widehat{\eta}_\alpha$ 
and $\widehat{\xi}_{i,\alpha}$ of $D$ by 
\begin{equation}\label{eq:two_chars}
  \widehat{\eta}_\alpha(\beta) = \exp(2\pi \sqrt{-1}\, \eta_\alpha(\beta)),
  \quad  
  \widehat{\xi}_{i,\alpha}(\beta) = \exp(2\pi \sqrt{-1}\,\xi_{i,\alpha}(\beta))
\end{equation}
for $\beta\in D$.

\begin{lemma}\label{lem:characters}
  Let $i \in I$. 

  \textup{(1)} The map 
  $\widehat{\eta} : C\to D^*;$ $\alpha\mapsto \widehat{\eta}_\alpha$
  is an epimorphism with kernel $D^\perp$.

  \textup{(2)} $\xi_{i,\alpha+\delta}=\xi_{i,\alpha}+\eta_{\delta}$ and 
  $\widehat{\xi}_{i,\alpha+\delta}=\widehat{\xi}_{i,\alpha}\,
  \widehat{\eta}_{\delta}$ for $\alpha$, $\delta\in C$.

  \textup{(3)} $\widehat{\xi}_{i,\alpha}=\widehat{\xi}_{i,\alpha'}$ if and only if 
  $\alpha \equiv \alpha' \pmod{D^\perp}$.
\end{lemma}

\begin{proof}
The assertion (1) holds by Corollary \ref{cor:bilinear}. 
The assertion (2) is clear by the definitions in \eqref{eq:two_maps} and \eqref{eq:two_chars}.
The assertion (3) follows from (1) and (2).
\end{proof}

Let 
\begin{equation}\label{eq:def_Uia+}
  U^{i,\alpha}
  = U\boxtimes_{W\otimes V} (W^{i,0}\otimes V^\alpha) 
  = \bigoplus_{\beta\in D} W^{i,\beta}\otimes V^{\alpha+\beta}
\end{equation}
for $i\in I$ and $\alpha \in C$.
The index $\alpha$ of $U^{i,\alpha}$ depends on the choice of 
a representative $W^{i,0}$ of the $D$-orbit $\mathcal{O}_W^i$. 
In fact, $U \boxtimes_{W \otimes V} (W^{i,\beta} \otimes V^\alpha) = U^{i, \alpha - \beta}$. 

\begin{theorem}\label{thm:all_modules}
  \textup{(1)} $U^{i,\alpha}$ is an irreducible $\widehat{\xi}_{i,\alpha}$-twisted $U$-module 
  for $(i,\alpha) \in I\times C$.

  \textup{(2)} For $(i,\alpha)$, $(i',\alpha') \in I\times C$, we have 
  $U^{i,\alpha}\cong U^{i',\alpha'}$ as $\chi$-twisted $U$-modules 
  for some $\chi \in D^*$ if and only if 
  $i=i'$ and $\alpha \equiv \alpha' \pmod{D_i}$.

  \textup{(3)} For $\chi \in D^*$ and $i \in I$, 
  there exists $\alpha \in C$ such that $U^{i,\alpha}$ is an irreducible $\chi$-twisted 
  $U$-module.

  \textup{(4)} Let $\chi \in D^*$. 
  Then any irreducible $\chi$-twisted $U$-module is isomorphic to $U^{i,\alpha}$ for some 
  $(i,\alpha) \in I\times C$.
\end{theorem}

\begin{proof}
Since $U$ is a $D$-graded simple current extension of $W\otimes V$, 
the assertion (1) follows from Lemma \ref{lem:free} and Theorem \ref{thm:induced_module}.

The uniqueness in Theorem \ref{thm:induced_module} implies that 
$U^{i,\alpha}\cong U^{i',\alpha'}$ as $\chi$-twisted $U$-modules for some $\chi \in D^*$  
if and only if 
$U^{i,\alpha}\cong U^{i',\alpha'}$ as $W\otimes V$-modules.
Hence 
$U^{i,\alpha}$ and $U^{i',\alpha'}$ are isomorphic $\chi$-twisted $U$-modules if and only if 
$W^{i,0}\otimes V^{\alpha}$ and $W^{i',0}\otimes V^{\alpha'}$ belong to the same 
$D$-orbit. 
Thus the assertion (2) holds.

As for (3), let $\chi\in D^*$ and $i\in I$.
By Lemma \ref{lem:characters}, there exists $\alpha\in C$ such that 
$\chi=\widehat{\xi}_{i,\alpha}$. 
Then $U^{i,\alpha}$ is an irreducible $\chi$-twisted $U$-module by (1).

Let $\chi\in D^*$ and let $M$ be an irreducible $\chi$-twisted $U$-module.
Since $W\otimes V$ is rational, we can take an irreducible $W\otimes V$-submodule 
$X\cong W^{i,\gamma}\otimes V^\alpha$ of $M$ for some $i\in I$, $\gamma\in D$, 
and $\alpha \in C$ 
by \eqref{eq:irr_W}.
Since the action \eqref{eq:diag_action} of $D$ on $\Irr(W\otimes V)$ is free, 
the $W\otimes V$-modules $(W^\beta\otimes V^\beta)\boxtimes_{W\otimes V} X$, $\beta\in D$, 
are all inequivalent. 
Therefore, by the irreducibility, $M$ is isomorphic to 
\[
  U \boxtimes_{W \otimes V} X
  = \bigoplus_{\beta \in D} W^{i,\gamma+\beta} \otimes V^{\alpha+\beta}
  = U^{i,\alpha-\gamma}
\] 
as $W \otimes V$-modules. 
Then the uniqueness in Theorem \ref{thm:induced_module} implies that 
$M \cong U^{i,\alpha-\gamma}$ as $\chi$-twisted $U$-modules.
This completes the proof of the assertion (4).
\end{proof}

For $i\in I$ and $\chi\in D^*$, we set 
\begin{equation*}
  C(i,\chi) = \{ \alpha \in C \mid \widehat{\xi}_{i,\alpha} = \chi\} .
\end{equation*}

\begin{lemma}\label{lem:onCchi}
  $(1)$ $C(i,\chi)$ is a coset of $D^\perp$ in $C$ for $i\in I$ and $\chi\in D^*$. 
  In particular, $C(0,1) = D^\perp$, where $1$ is the principal character of $D$.

  $(2)$ $C(i,\widehat{\eta}_\alpha)=\alpha+C(i,1)$ for $\alpha \in C$.
\end{lemma}

\begin{proof}
Let $i \in I$. 
By (1) and (2) of Lemma \ref{lem:characters}, we see that 
$D^* = \{ \widehat{\xi}_{i,\alpha} \mid \alpha \in C \}$. 
Then the assertion (1) follows from (3) of Lemma \ref{lem:characters}.
The assertion (2) follows from (2) of Lemma \ref{lem:characters}.
\end{proof}

For $\chi \in D^\ast$, let 
\begin{equation}\label{eq:R_chi}
R(\chi) = \{ (i, \alpha) \mid i \in I, \alpha \in C(i,\chi) \}. 
\end{equation}
Then any irreducible $\chi$-twisted $U$-module is isomorphic to $U^{i,\alpha}$ 
for some $(i,\alpha) \in R(\chi)$ by Theorem \ref{thm:all_modules}. 
For $(i,\alpha)$, $(i',\alpha') \in R(\chi)$, we write 
\begin{equation}\label{eq:approx}
(i,\alpha) \approx (i',\alpha')
\end{equation}
if $U^{i,\alpha} \cong U^{i',\alpha'}$ as $\chi$-twisted $U$-modules. 
We have $(i,\alpha) \approx (i',\alpha')$ if and only if 
$i=i'$ and $\alpha \equiv \alpha' \pmod{D_i}$ by Theorem \ref{thm:all_modules}. 
The relation $(i,\alpha) \approx (i',\alpha')$ is an equivalence relation on $R(\chi)$, 
and the equivalence classes of irreducible $\chi$-twisted $U$-modules are 
indexed by $R(\chi)/{\approx}$.
In particular,
\[
  \Irr(U) = \{ U^{i,\alpha} \mid (i,\alpha) \in R(1)/{\approx} \}.
\]

\begin{theorem}\label{thm:number_of_irred}
  For $\chi\in D^*$, the number of inequivalent irreducible $\chi$-twisted $U$-modules is 
  $\abs{C}\abs{\Irr(W)}/\abs{D}^2$.
\end{theorem}

\begin{proof}
The number of inequivalent irreducible $\chi$-twisted $U$-modules is 
$\abs{R(\chi)/{\approx}}$, which is equal to
\[
  \sum_{i\in I} [D^\perp:D_i]
  = \dfrac{\abs{D^\perp}}{\abs{D}} \sum_{i\in I} [D:D_i]
  = \dfrac{[C:D]}{\abs{D}} \sum_{i\in I} \abs{\mathcal{O}^i_W}
  = \dfrac{\abs{C}}{\abs{D}^2} \abs{\Irr(W)}
\]
by Lemmas \ref{lem:Di_D} and \ref{lem:onCchi}, and \eqref{eq:OW}. 
\end{proof}

\begin{corollary}\label{cor:count}
  $\abs{\Irr(U)}=\abs{C}\abs{\Irr(W)}/\abs{D}^2$.
\end{corollary}

\begin{remark}
  Since every irreducible $U$-module is $\chi$-stable for $\chi \in D^*$, 
  one can also show Corollary \ref{cor:count} by using the modular invariance of 
  twisted modules \cite[Theorem 10.2]{DLM2000}.
\end{remark}

\subsection{Fusion rules}

We consider the relationship between the fusion rule of $W$-modules and 
that of $U$-modules. 
For $\chi \in D^\ast$, let 
\begin{equation*}
  T(\chi) = \{ (i, \alpha, \beta) \mid i \in I, \alpha \in C(i,\chi), \beta \in D \}, 
\end{equation*}
that is, $T(\chi) = R(\chi) \times D$ by \eqref{eq:R_chi}. 
For $(i, \alpha, \beta)$, $(i', \alpha', \beta') \in T(\chi)$, we write
\[
  (i, \alpha, \beta) \sim (i', \alpha', \beta')
\]
if $W^{i,\beta} \otimes V^{\alpha+\beta} \cong W^{i',\beta'} \otimes V^{\alpha'+\beta'}$ 
as $W \otimes V$-modules. 
We have
$(i, \alpha, \beta) \sim (i', \alpha', \beta')$ if and only if 
$i = i'$, $\beta \equiv \beta' \pmod{D_i}$, and $\alpha+\beta = \alpha'+\beta'$ by 
Lemma \ref{lem:irr_W}. 
The relation $(i, \alpha, \beta) \sim (i', \alpha', \beta')$ 
is an equivalence relation on $T(\chi)$, and $\Irr(W \otimes V)$ is indexed by 
$T(\chi)/{\sim}$, $\chi \in D^\ast$.
\[
  \Irr(W \otimes V) = 
  \{ W^{i,\beta} \otimes V^{\alpha+\beta} \mid (i, \alpha, \beta) \in T(\chi)/{\sim}, \chi \in D^\ast \}.
\]

For $(i,\alpha,\beta) \in T(\chi)$, we have 
$U \boxtimes_{W \otimes V} (W^{i,\beta} \otimes V^{\alpha+\beta}) = U^{i,\alpha}$ 
is an irreducible $\chi$-twisted $U$-module. 
Hence for $(i',\alpha',\beta') \in T(\chi)$, we see that 
\[
  U \boxtimes_{W \otimes V} (W^{i,\beta} \otimes V^{\alpha+\beta}) 
  \cong U \boxtimes_{W \otimes V} (W^{i',\beta'} \otimes V^{\alpha'+\beta'})
\]
as $\chi$-twisted $U$-modules if and only if $(i,\alpha) \approx (i',\alpha')$. 
Note also that $(i, \alpha, \beta) \sim (i', \alpha', \beta')$ implies 
$(i,\alpha) \approx (i',\alpha')$. 

Now, recall the notation 
in Section \ref{subsec:simple_current_ext}. 
Let $\CC_{W \otimes V}$ be the category of $W \otimes V$-modules. 
We denote by $\CC_{W \otimes V}^0$ the full subcategory of $\CC_{W \otimes V}$ 
consisting of the objects $X$ of $\CC_{W \otimes V}$ such that 
$U \boxtimes_{W \otimes V} X$ is an ordinary $U$-module, 
see \cite[Definition 2.66]{CKM2017}. 
Let
\[
  \Irr^0(W \otimes V) 
  = \{ W^{i,\beta} \otimes V^{\alpha+\beta} \mid (i,\alpha,\beta) \in T(1)/{\sim} \},
\]
which constitutes the simple objects of $\CC_{W \otimes V}^0$,  
see also \cite[Proposition 2.65]{CKM2017}.
Let $\CC_U$ be the category of $U$-modules. 
Then Theorem \ref{thm:functor_F} implies the following theorem. 

\begin{theorem}\label{thm:functor_F_on_CWV0}
  The category $\CC_{W \otimes V}^0$ is a $\C$-linear additive braided monoidal category 
  with structures induced from $\CC_{W \otimes V}$, 
  and the functor $F : \CC_{W \otimes V}^0 \to \CC_U$; 
  $X \mapsto U \boxtimes_{W \otimes V} X$ is a braided tensor functor.
\end{theorem}

Since the category $\CC_{W \otimes V}^0$ is closed under the fusion product 
$\boxtimes_{W \otimes V}$,  
we have 
\begin{equation*}
\begin{split}
  & (W^{i_1,\beta_1} \otimes V^{\alpha_1+\beta_1}) \boxtimes_{W \otimes V} 
  (W^{i_2,\beta_2} \otimes V^{\alpha_2+\beta_2})\\
  &\qquad = \sum_{(i_3,\alpha_3,\beta_3) \in T(1)/{\sim}} 
  n(i_3,\alpha_3,\beta_3) W^{i_3,\beta_3} \otimes V^{\alpha_3+\beta_3}
\end{split}
\end{equation*}
for $(i_1,\alpha_1,\beta_1)$, $(i_2,\alpha_2,\beta_2) \in T(1)$, where 
\begin{equation*}
  n(i_3,\alpha_3,\beta_3) = \dim
  I_{W\otimes V} 
  \binom{W^{i_3,\beta_3}\otimes V^{\alpha_3+\beta_3}}
  {W^{i_1,\beta_1}\otimes V^{\alpha_1+\beta_1}\quad 
  W^{i_2,\beta_2}\otimes V^{\alpha_2+\beta_2}}.
\end{equation*}

Since $V^{\alpha_1+\beta_1} \boxtimes_V V^{\alpha_2+\beta_2} 
= V^{\alpha_1+\beta_1+\alpha_2+\beta_2}$, we actually have
\begin{equation}\label{eq;fusion_WV-2}
\begin{split}
  & (W^{i_1,\beta_1} \otimes V^{\alpha_1+\beta_1}) \boxtimes_{W \otimes V} 
  (W^{i_2,\beta_2} \otimes V^{\alpha_2+\beta_2})\\
  &\qquad = 
  \sum_{\substack{(i_3,\alpha_3,\beta_3) \in T(1)/{\sim}\\ 
  \alpha_3+\beta_3 = \alpha_1+\beta_1+\alpha_2+\beta_2}} 
  n(i_3,\alpha_3,\beta_3) W^{i_3,\beta_3} \otimes V^{\alpha_3+\beta_3}.
\end{split}
\end{equation}
In fact, 
\begin{equation*}
\begin{split}
  &I_{W\otimes V} 
  \binom{W^{i_3,\beta_3}\otimes V^{\alpha_3+\beta_3}}
  {W^{i_1,\beta_1}\otimes V^{\alpha_1+\beta_1}\quad 
  W^{i_2,\beta_2}\otimes V^{\alpha_2+\beta_2}}\\
  &\qquad \cong I_W \binom{W^{i_3,\beta_3}}{W^{i_1,\beta_1}\quad W^{i_2,\beta_2}}
  \otimes 
  I_V\binom{V^{\alpha_3+\beta_3}}{V^{\alpha_1+\beta_1}\quad V^{\alpha_2+\beta_2}}
\end{split}
\end{equation*}
by Proposition \ref{prop:tensor_prod_fusion}. 
Let 
\begin{equation*}
  n(i_3, \beta_3) = \dim I_W \binom{W^{i_3,\beta_3}}{W^{i_1,\beta_1}\quad W^{i_2,\beta_2}}
\end{equation*}
Then $n(i_3,\alpha_3,\beta_3) = n(i_3, \beta_3)$ 
if $\alpha_3+\beta_3 = \alpha_1+\beta_1+\alpha_2+\beta_2$,  
and $n(i_3,\alpha_3,\beta_3) = 0$ otherwise.

We fix $(i_1,\alpha_1,\beta_1)$, $(i_2,\alpha_2,\beta_2) \in T(1)$. 
For given $i_3 \in I$ and $\beta_3 \in D$, there exists an element 
$\alpha_3 \in C(i_3,1)$ which satisfies the conditions 
\begin{equation}\label{eq:condition_on_a3b3}
  (i_3,\alpha_3,\beta_3) \in T(1)/{\sim} \quad \text{and} \quad 
  \alpha_3+\beta_3 = \alpha_1+\beta_1+\alpha_2+\beta_2
\end{equation}
if and only if $\alpha_1 + \beta_1 + \alpha_2 + \beta_2 - \beta_3 \in C(i_3,1)$. 
Set
\begin{equation}\label{eq:P}
  P = \{ (i_3,\beta_3) \in I \times D \mid 
  \alpha_1 + \beta_1 + \alpha_2 + \beta_2 - \beta_3 \in C(i_3,1)\}.
\end{equation}

Since $C(i_3,1) = \alpha_3 + D^\perp$ by Lemma \ref{lem:onCchi}, 
the condition $\alpha_1+\beta_1+\alpha_2+\beta_2-\beta_3 \in C(i_3,1)$ is equivalent 
to the condition $\alpha_1+\beta_1+\alpha_2+\beta_2-\alpha_3-\beta_3 \in D^\perp$. 
Since $D_{i_3} \subset D^\perp$ by Lemma \ref{lem:Di_D}, 
$P$ is a union of equivalence classes 
with respect to the equivalence relation $\sim$ defined in \eqref{eq:sim}. 
We also note that $(i,\alpha,\beta) \sim (i',\alpha',\beta')$ in $T(1)$ implies 
$(i,\beta) \sim (i',\beta')$ in $I \times D$. 
Therefore, \eqref{eq;fusion_WV-2} can be written as
\begin{equation*}
\begin{split}
  & (W^{i_1,\beta_1} \otimes V^{\alpha_1+\beta_1}) \boxtimes_{W \otimes V} 
  (W^{i_2,\beta_2} \otimes V^{\alpha_2+\beta_2})\\
  &\qquad = 
  \sum_{(i_3,\beta_3) \in P/{\sim}}
  n(i_3,\beta_3) W^{i_3,\beta_3} \otimes V^{\alpha_1+\beta_1+\alpha_2+\beta_2},
\end{split}
\end{equation*}
which implies that
\begin{equation}\label{eq:fusion_W-1}
  W^{i_1,\beta_1} \boxtimes_W W^{i_2,\beta_2} 
  = \sum_{(i_3,\beta_3) \in P/{\sim}} 
  \dim I_W \binom{W^{i_3,\beta_3}}{W^{i_1,\beta_1}\quad W^{i_2,\beta_2}} 
  W^{i_3,\beta_3}.
\end{equation}

Since $U \boxtimes_{W \otimes V} (W^{i,\beta} \otimes V^{\alpha+\beta}) = U^{i,\alpha}$, 
it follows from Theorem \ref{thm:functor_F_on_CWV0} and \eqref{eq;fusion_WV-2} that 
\begin{equation}\label{eq;fusion_WV-3}
U^{i_1,\alpha_1} \boxtimes_U U^{i_2,\alpha_2} 
= \sum_{\substack{(i_3,\alpha_3,\beta_3) \in T(1)/{\sim}\\ 
\alpha_3+\beta_3 = \alpha_1+\beta_1+\alpha_2+\beta_2}} 
n(i_3,\alpha_3,\beta_3) U^{i_3,\alpha_3}.
\end{equation}

For given $i_3 \in I$ and $\alpha_3 \in C(i_3,1)$, there exists an element $\beta_3 \in D$ 
which satisfies the conditions 
\eqref{eq:condition_on_a3b3} 
if and only if $\alpha_1 + \alpha_2 - \alpha_3 \in D$. 
Set
\begin{equation}\label{eq:Q}
  Q = \{ (i_3,\alpha_3)\in I \times C \mid 
  \alpha_3 \in C(i_3,1),~ \alpha_1+\alpha_2-\alpha_3\in D\}.
\end{equation}
Then $Q$ is a union of equivalence classes of the equivalence relation 
$\approx$ defined in \eqref{eq:approx}. 
Since $(i,\alpha,\beta) \sim (i',\alpha',\beta')$ in $T(1)$ implies 
$(i,\alpha) \approx (i',\alpha')$ in $R(1)$, we can write \eqref{eq;fusion_WV-3} as 
\begin{equation}\label{eq;fusion_U-1}
  U^{i_1,\alpha_1} \boxtimes_U U^{i_2,\alpha_2} 
  = \sum_{(i_3,\alpha_3) \in Q/{\approx}} 
  n(i_3, \alpha_1+\beta_1+\alpha_2+\beta_2-\alpha_3) U^{i_3,\alpha_3}.
\end{equation}

For fixed $(i_1,\alpha_1,\beta_1)$, $(i_2,\alpha_2,\beta_2) \in T(1)$, 
define a map $\psi$ by
\begin{equation}\label{eq:psi}
  \psi: P \to Q;\quad
  (i_3,\beta_3) \mapsto (i_3,\alpha_1+\beta_1+\alpha_2+\beta_2-\beta_3).
\end{equation}
The map $\psi$ is a bijection, and its inverse is 
\begin{equation}\label{eq:psi_inv}
  \psi^{-1}: Q\to P;\quad 
  (i_3,\alpha_3) \mapsto (i_3,\alpha_1+\beta_1+\alpha_2+\beta_2-\alpha_3).
\end{equation}

Since $(i,\beta) \sim (i',\beta')$ if and only if 
$\psi(i,\beta) \approx \psi(i',\beta')$ for $(i,\beta)$, $(i',\beta') \in P$, 
the map $\psi$ induces a bijection between $P/{\sim}$ and $Q/{\approx}$. 
Then we see from \eqref{eq;fusion_U-1} that
\begin{equation}\label{eq:fusion_U-2}
  U^{i_1,\alpha_1} \boxtimes_U U^{i_2,\alpha_2} 
  = \sum_{(i_3,\alpha_3) \in Q/{\approx}} 
  \dim I_U\binom{U^{i_3,\alpha_3}}{U^{i_1,\alpha_1} \quad U^{i_2,\alpha_2}} 
  U^{i_3,\alpha_3}
\end{equation}
with
\begin{equation}\label{eq:fusion_rule_U_W}
  \dim I_U\binom{U^{i_3,\alpha_3}}{U^{i_1,\alpha_1} \quad U^{i_2,\alpha_2}} 
  = \dim I_W \binom{W^{\psi^{-1}(i_3,\alpha_3)}}
   {W^{i_1,\beta_1} \quad W^{i_2,\beta_2}}
\end{equation}
for $(i_3,\alpha_3) \in Q/{\approx}$.
Using the bijection $\psi$, we also have
\begin{equation}\label{eq:fusion_rule_W_U}
    \dim I_W\binom{W^{i_3,\beta_3}}{W^{i_1,\beta_1} \quad W^{i_2,\beta_2}}
    = \dim I_U\binom{U^{\psi(i_3,\beta_3)}}
  {U^{i_1,\alpha_1} \quad U^{i_2,\alpha_2}}
\end{equation}
for $(i_3,\beta_3) \in P/{\sim}$.
Therefore, we obtain the following theorem from 
\eqref{eq:fusion_W-1} and \eqref{eq:fusion_U-2} 
by the bijection $\psi$ between $P/{\sim}$ and $Q/{\approx}$. 

\begin{theorem}\label{thm:fusion_products_4_U_W}
  Let $(i_1,\alpha_1,\beta_1)$, $(i_2,\alpha_2,\beta_2) \in T(1)$.
  Then the fusion product $W^{i_1,\beta_1} \boxtimes_W W^{i_2,\beta_2}$ 
  and the fusion product $U^{i_1,\alpha_1} \boxtimes_U U^{i_2,\alpha_2}$ 
  are related as follows.
\begin{equation*}
\begin{split}
  W^{i_1,\beta_1} \boxtimes_W W^{i_2,\beta_2} 
  &= \sum_{(i_3,\alpha_3) \in Q/{\approx}} \dim I_U 
  \binom{U^{i_3,\alpha_3}}{U^{i_1,\alpha_1}\quad U^{i_2,\alpha_2}} W^{\psi^{-1}(i_3,\alpha_3)},\\
  U^{i_1,\alpha_1} \boxtimes_U U^{i_2,\alpha_2} 
  &= \sum_{(i_3,\beta_3) \in P/{\sim}} \dim I_W 
  \binom{W^{i_3,\beta_3}}{W^{i_1,\beta_1} \quad W^{i_2,\beta_2}} U^{\psi(i_3,\beta_3)},
\end{split}
\end{equation*}
where $P$, $Q$, $\psi$, and $\psi^{-1}$ are defined as in 
\eqref{eq:P}, \eqref{eq:Q}, \eqref{eq:psi} 
and \eqref{eq:psi_inv}, respectively.
\end{theorem}

\begin{corollary}\label{cor:SCU}
  Let $I_{\mathrm{sc}} = \{ i\in I \mid W^{i,0}\in \SC{W}\}$.
  Then
  \[
    \SC{U} = \{ U^{i,\alpha} \mid i \in I_{\mathrm{sc}},~ \alpha \in C(i,1)\},
  \]
  and $\abs{\SC{U}} = \abs{C}\abs{\SC{W}}/\abs{D}^2$.
\end{corollary}

\begin{proof}
Any simple current $W$-module is isomorphic to $W^{i,\beta}$ for some 
$i \in I_{\mathrm{sc}}$ and $\beta \in D$. 
Thus any simple current $U$-module is isomorphic to $U^{i,\alpha}$ 
for some $i \in I_{\mathrm{sc}}$ and $\alpha \in C(i,1)$ 
by Theorem \ref{thm:fusion_products_4_U_W}.
We have $D_i = 0$ for $i \in I_{\mathrm{sc}}$. 
Therefore, $\abs{\SC{W}} = \abs{I_{\mathrm{sc}}} \abs{D}$ and
$\abs{\SC{U}} = \abs{I_{\mathrm{sc}}} \abs{D^\perp}$. 
Thus the assertion holds.
\end{proof}

Let 
\begin{equation}\label{eq:Ugamma}
  U^\gamma = U^{0,\gamma} = \bigoplus_{\beta \in D} W^\beta \otimes V^{\gamma+\beta}
\end{equation}
for $\gamma \in D^\perp = C(0,1)$.

\begin{corollary}\label{cor:canonical_SC}
  $\{ U^\gamma \mid \gamma \in D^\perp\}$ is a $D^\perp$-graded set of 
  simple current $U$-modules.
  Furthermore, $U^\gamma \boxtimes_U U^{i,\alpha}=U^{i,\alpha+\gamma}$ for 
  $i \in I$ and $\alpha \in C(i,1)$.
\end{corollary}

\begin{proof}
All $U^\gamma$, $\gamma\in D^\perp$ are inequivalent simple current $U$-modules 
by Corollary \ref{cor:SCU}. 
Since $W^\beta \boxtimes_W W^{i,\delta}=W^{i,\delta+\beta}$, 
the assertion holds by Theorem \ref{thm:fusion_products_4_U_W}.
\end{proof}

\begin{remark}\label{rem:underSC}
  Let $X\in \Irr(W)$ and $M$, $N\in \Irr(U)$.
  Suppose $X\subset M$.
  Then it follows from Corollary \ref{cor:canonical_SC} that $N$ contains 
  an irreducible $W$-submodule isomorphic to $X$ if and only if 
  $N\cong U^{\gamma}\boxtimes_U M$ for some $\gamma\in D^\perp$.
\end{remark}

The next lemma follows from Lemma \ref{lem:Atensor}, \eqref{eq:def_Wib+}, and 
Corollary \ref{cor:canonical_SC}.

\begin{lemma}\label{lem:fusion_shift}
  Let $i_p \in I$, $p = 1,2,3$.

  \textup{(1)} 
  For $\beta_p$, $\delta_p \in D$, $p = 1,2,3$, we have
  \[
    \dim I_W \binom{W^{i_3,\beta_3+\delta_3}}{W^{i_1,\beta_1+\delta_1}\quad W^{i_2,\beta_2+\delta_2}}
    = \dim I_W \binom{W^{i_3,\beta_3-\delta_1-\delta_2+\delta_3}}
   {W^{i_1,\beta_1}\quad W^{i_2,\beta_2}}.
  \]

  \textup{(2)} 
  For $\alpha_p \in C(i_p,1)$ and $\gamma_p \in D^\perp$, $p = 1,2,3$, 
  we have
  \[
    \dim I_U \binom{U^{i_3,\alpha_3+\gamma_3}}
   {U^{i_1,\alpha_1+\gamma_1} \quad U^{i_2,\alpha_2+\gamma_2}}
    = \dim I_U \binom{U^{i_3,\alpha_3-\gamma_1-\gamma_2+\gamma_3}}
   {U^{i_1,\alpha_1}\quad U^{i_2,\alpha_2}}.
  \]
\end{lemma}

For later use, 
we rewrite \eqref{eq:fusion_rule_U_W} and \eqref{eq:fusion_rule_W_U} 
in the following form.

\begin{theorem}\label{thm:fusion_rules2}
  Let $X^p \in \Irr(W)$ and $M^p\in \Irr(U)$, and suppose 
  $X^p\otimes V^{\alpha_p}\subset M^p$ for some $\alpha_p \in C$, $p=1,2,3$.
  Let $\gamma = \alpha_1+\alpha_2-\alpha_3$.

  \textup{(1)} If $\gamma \in D$ then
  \[
    \dim I_U\binom{M^3}{M^1 \quad M^2} 
    = \dim I_W\binom{W^{\gamma} \boxtimes_W X^3}{X^1 \quad X^2},
  \]
  and otherwise 
  \[
    \dim I_U\binom{M^3}{M^1 \quad M^2} = 0.
  \]

  \textup{(2)} If $\gamma \in D^\perp$ then 
  \[
    \dim I_W\binom{X^3}{X^1 \quad X^2} 
    = \dim I_U\binom{U^{\gamma} \boxtimes_U M^3}{M^1 \quad M^2},
  \]
  and otherwise 
  \[
    \dim I_W\binom{X^3}{X^1 \quad X^2} =0.
  \]
\end{theorem}

Indeed, since $W^{i,0}$ is an arbitrary chosen representative of the $D$-orbit
$\mathcal{O}^i_W$ in $\Irr(W)$, we may take 
$X^p = W^{i_p,0}$ for some $i_p \in I$. 
Then $M^p = U^{i_p,\alpha_p}$, $p = 1,2,3$.  
Apply \eqref{eq:fusion_rule_U_W} and \eqref{eq:fusion_rule_W_U} 
with $\beta_1 = \beta_2 = \beta_3 = 0$. 
Then Theorem \ref{thm:fusion_rules2} follows from 
\eqref{eq:def_Wib+}, Corollary \ref{cor:canonical_SC},  and 
Lemma \ref{lem:fusion_shift}.

\subsection{$R(U)$ versus $R(W)$}\label{subsec:RU_RW}

In this subsection, we discuss the relationship between the fusion algebras $R(U)$ and $R(W)$. 
In view of Corollary \ref{cor:canonical_SC} 
we define an action of $D^\perp$ on $\Irr(U)$ by 
\begin{equation}\label{eq:Dperpaction}
  M \mapsto U^\gamma \boxtimes_U M
\end{equation}
for $\gamma \in D^\perp$ and $M \in \Irr(U)$.
Then for each $i \in I$, the set 
\begin{equation}\label{eq:OUi}
  \mathcal{O}^i_U = \{ U^{i,\alpha} \mid \alpha \in C(i,1)\}
\end{equation}
forms a $D^\perp$-orbit by Lemma \ref{lem:onCchi}, 
and we obtain the $D^\perp$-orbit decomposition
\begin{equation}\label{eq:OU}
  \Irr(U) = \bigcup_{i\in I} \mathcal{O}^i_U .
\end{equation}
That is, we can use the same index set $I$ 
as in the $D$-orbit decomposition \eqref{eq:OW} of $\Irr(W)$. 
Note that $\mathcal{O}^0_U = \{ U^\gamma \mid \gamma \in D^\perp\}$ 
as $C(0,1) = D^\perp$.

We consider another description of the $D^\perp$-orbit $\mathcal{O}^i_U$ in $\Irr(U)$.
For each $i \in I$, we pick $M^{i,0} \in \mathcal{O}_U^i$ and fix it, 
where we choose $M^{0,0} = U$.
Moreover, we set
\begin{equation*}
  M^{i,\gamma} = U^\gamma \boxtimes_U M^{i,0}
\end{equation*}
for $\gamma \in D^\perp$.
Then $\mathcal{O}^i_U = \{ M^{i,\gamma} \mid \gamma \in D^\perp\}$. 
Let
\begin{equation*}
  (D^\perp)_i = \{ \gamma \in D^\perp \mid U^\gamma \boxtimes_U M \cong M  
  \mbox{ for } M \in \mathcal{O}^i_U\}.
\end{equation*}

Since $D^\perp$ is abelian, $(D^\perp)_i$ is the stabilizer of $M$ for 
any $M \in \mathcal{O}^i_U$. 
Therefore, the length of the orbit $\mathcal{O}^i_U$ is $[D^\perp : (D^\perp)_i]$, and 
\[
  \Irr(U) = \{ M^{i,\gamma} \mid i \in I, \gamma \in D^\perp/(D^\perp)_i\}.
\]

\begin{lemma}\label{lem:DiperpDi}
  $(D^\perp)_i = D_i$ for $i\in I$.
\end{lemma}

\begin{proof} 
The asserton follows from Lemma \ref{lem:Di_D}, (2) of Theorem \ref{thm:all_modules}, 
Corollary \ref{cor:canonical_SC}, and \eqref{eq:OUi}.
\end{proof}

Next, we describe $\Irr(W)$.
By (3) of Theorem \ref{thm:all_modules}, 
any irreducible $W$-module appears as a submodule of 
an irreducible $U$-module.
Let $M^{i,\gamma} \in \mathcal{O}^i_U$ with $\gamma \in D^\perp$.
Then $M^{i,\gamma} \cong U^{i,\alpha}$ for some $\alpha \in C(i,1)$, and 
this $\alpha$ is uniquely determined modulo $D_i = (D^\perp)_i$.  
By the structure \eqref{eq:def_Uia+} of irreducible $U$-modules, 
there exists a coset $\lambda(i,\gamma)+D \in C/D$ such that 
$M^{i,\gamma}$ has a decomposition of the form
\begin{equation}\label{eq:Xigammad}
  M^{i,\gamma} = \bigoplus_{\delta \in \lambda(i,\gamma)+D} X^{i,\gamma,\delta} \otimes V^\delta
\end{equation}
as a $W\otimes V$-module, where $X^{i,\gamma,\delta} \in \Irr(W)$ 
is the multiplicity of $V^\delta$ in $M^{i,\gamma}$. 
The set $\{ X^{i,\gamma,\delta} \mid \delta\in \lambda(i,\gamma)+D\}$ 
is independent of $\gamma \in D^\perp$, and it coincides with 
the $D$-orbit $\mathcal{O}^i_W = \{ W^{i,\beta} \mid \beta \in D\}$ in $\Irr(W)$ 
in \eqref{eq:OW}. 
In other words, 
the $D$-orbit $\mathcal{O}^i_W$ is uniquely determined by the $D^\perp$-orbit 
$\mathcal{O}^i_U$ as we mentioned in Remark \ref{rem:underSC}.
Moreover, since $D_i = (D^\perp)_i$, 
we have $X^{i,\gamma,\delta} \cong X^{i,\gamma,\delta'}$ as $W$-modules
if and only if $\delta \equiv \delta' \pmod{(D^\perp)_i}$. 
Therefore, we obtain another description of $\Irr(W)$ 
as follows.

\begin{theorem}\label{thm:irrWbyU}
  Define a $D^\perp$-graded set of simple currents 
  $\{ U^\gamma \mid \gamma \in D^\perp\} \subset \SC{U}$ as in \eqref{eq:Ugamma}, 
  and consider the $D^\perp$-orbit decomposition of $\Irr(U)$ as in \eqref{eq:OU}.
  Pick $M^{i,0}\in \mathcal{O}^i_U$ for each $i\in I$,  
  and collect $X^{i,0,\delta} \in \Irr(W)$ in the decomposition \eqref{eq:Xigammad} of 
  $M^{i,0}$. 
  Then $\Irr(W) = \{ X^{i,0,\delta} \mid i \in I, \delta \in \lambda(i,0)+D/(D^\perp)_i\}$.
\end{theorem}

The above theorem and (2) of Theorem \ref{thm:fusion_rules2} imply that 
the structure of $R(W)$ is determined by those of $R(U)$ and $R(V)$. 

We summarize the outcomes of Hypothesis \ref{hypo:W_V} as follows. 
Let $V$, $W$, $C$, $D$, and $U$ be as in Hypothesis \ref{hypo:W_V}. 
We have a $D$-graded set of  simple current $W$-modules 
$\{ W^\beta \mid \beta \in D\}$ 
as in Hypothesis \ref{hypo:W_V}, and a $D^\perp$-graded set of 
simple current $U$-modules $\{ U^\gamma \mid \gamma \in D^\perp\}$ 
as in \eqref{eq:Ugamma}.
Consider an action of $D$ on $\Irr(W)$ and an action of $D^\perp$ on $\Irr(U)$ as in 
\eqref{eq:Daction+} and \eqref{eq:Dperpaction}, respectively.
Let $\mathbb{O}_W$ be the set of $D$-orbits in $\Irr(W)$, and let $\mathbb{O}_U$ be 
the set of $D^\perp$-orbits in $\Irr(U)$, respectively. 
Then $\abs{\mathbb{O}_W} = \abs{\mathbb{O}_U} = \abs{I}$. 
We define a map $\Phi : \mathbb{O}_U \to \mathbb{O}_W$ as follows.
Let $\mathcal{O}\in \mathbb{O}_U$ and $M \in \mathcal{O}$.
Define $\Phi(\mathcal{O})$ to be the set of equivalence classes of irreducible 
$W$-submodules of $M$.
Then $\Phi(\mathcal{O})$ is independent of the choice of $M \in \mathcal{O}$, 
and the map $\Phi$ is well-defined.
Moreover, $\mathcal{O} \mapsto \Phi(\mathcal{O})$ 
is a bijection between $\mathbb{O}_U$ and $\mathbb{O}_W$.

The inverse $\Psi : \mathbb{O}_W \to \mathbb{O}_U$ of $\Phi$ is described as follows.
Let $\mathcal{O}' \in \mathbb{O}_W$ and $X \in \mathcal{O'}$.
Then the $\Q/\Z$-valued map $\beta\mapsto b_W(W^\beta,X)$ is linear on $D$ 
by Proposition \ref{prop:bV-bilinear-rev}.
Since $b_V$ is non-degenerate on $C$ by Corollary \ref{cor:bilinear}, 
we can find $\lambda\in C$ such that 
$b_W(W^\beta,X)+b_V(V^\beta,V^{\lambda})=0$ for all $\beta\in D$.
Such a $\lambda$ is unique modulo $D^\perp$. 
By (1) of Theorem \ref{thm:all_modules}, 
$U\boxtimes_{W\otimes V} (X\otimes V^\alpha)$ is an irreducible untwisted $U$-module 
for $\alpha \in \lambda +D^\perp$. 
We define $\Psi(\mathcal{O}')$ to be the set of equivalence classes of 
$U \boxtimes_{W \otimes V} (X \otimes V^\alpha)$,  $\alpha \in \lambda + D^\perp$.
Although $\lambda$ depends on the choice of $X \in \mathcal{O'}$, 
the resulting $D^\perp$-orbit $\Psi(\mathcal{O}')$ is independent of the choice of $X$, 
and it is uniquely determined by $\mathcal{O}'$, 
see the comment just after \eqref{eq:def_Uia+} and Remark \ref{rem:underSC}. 
Thus the map $\Psi$ is well-defined, and  
$\mathcal{O} = \Psi(\mathcal{O}')$ if $\mathcal{O}' = \Phi(\mathcal{O})$.

For $\mathcal{O} \in \mathbb{O}_U$ and $\mathcal{O}' \in \mathbb{O}_W$, 
set
\begin{equation*}
\begin{split}
  (D^\perp)_{\mathcal{O}}
  &= \{ \gamma \in D^\perp \mid U^\gamma \boxtimes_U M \cong M \mbox{ for } 
  M \in \mathcal{O}\},\\
  D_{\mathcal{O}'}
  &= \{ \beta \in D \mid W^\beta \boxtimes_W X \cong X \mbox{ for } 
  X \in \mathcal{O}'\}.
\end{split}
\end{equation*}
Then $D_{\Phi(\mathcal{O})} = (D^\perp)_{\mathcal{O}}$ 
and $(D^\perp)_{\Psi(\mathcal{O}')} = D_{\mathcal{O}'}$ 
by Lemma \ref{lem:DiperpDi}, and we have
\[
  \frac{\abs{\mathcal{O}}}{\abs{\Phi(\mathcal{O})}}
  = \frac{\abs{\Psi(\mathcal{O}')}}{\abs{\mathcal{O}'}}
  = \frac{\abs{D^\perp}}{\abs{D}}.
\]
This equation essentially explains Corollary \ref{cor:count}.

The equivalence classes $\Irr(U)$ and $\Irr(W)$ are mutually described by 
Theorems \ref{thm:all_modules} and \ref{thm:irrWbyU}, and their fusion rules are mutually 
described by Theorem \ref{thm:fusion_rules2}.
Thus we can completely describe the structures of $R(U)$ and $R(W)$ each other 
based on the duality between $D$ and $D^\perp$ in the quadratic space $(C,q_V)$ 
associated with $R(V)$.

\begin{remark}
  Let $\Irr(U;\chi)$ be the set of equivalence classes of irreducible $\chi$-twisted
  $U$-modules for $\chi \in D^*$.
  If $\chi = \eta_\alpha$ for $\alpha\in C$, then by Theorem \ref{thm:all_modules} and 
  (2) of Lemma \ref{lem:onCchi} there is a bijection between $\Irr(U)$ 
  and $\Irr(U;\chi)$ given by 
  \[
    \Theta_\alpha : \Irr(U) \to \Irr(U;\chi);\quad 
    \Theta_\alpha(M) = M \boxtimes_{W\otimes V} (W \otimes V^\alpha)
  \]
  for $M \in \Irr(U)$.
  We can similarly consider an action of $D^\perp$ on $\Irr(U;\chi)$ and obtain
  a $D^\perp$-orbit decomposition of $\Irr(U;\chi)$.
  Then the map $\Theta_\alpha$ induces a bijection between 
  $\mathbb{O}_U$ and the set of $D^\perp$-orbits in $\Irr(U;\chi)$. 
  Hence we have a bijective correspondence between 
  $\mathbb{O}_W$ and the set of $D^\perp$-orbits in $\Irr(U;\chi)$ 
  for any $\chi \in D^*$ as well. 
\end{remark}

\section*{Acknowledgments}

We would like to thank the referee of the previous version of the paper 
for suggesting us the hexagon axioms 
in the proof of Proposition \ref{prop:bV-bilinear-rev} 
and the functor $F$ in Theorem \ref{thm:functor_F}, 
which enable us to remove an assumption in the previous version that 
the conformal weight of any irreducible module for a vertex operator algebra 
is positive except for the adjoint module. 
We also thank 
Dra\v{z}en Adamovi\'c, Thomas  Creutzig, and Ching Hung Lam 
for helpful advice. 
Part of this work was done while the first author was staying at 
Institute of Mathematics, Academia Sinica, Taiwan as a visiting scholar from 
May 1, 2017 through April 30, 2018.
The second author visited the institute in August and September, 2017 and 
February and March, 2018. 
They are grateful to the institute.
The second author was partially supported by 
JSPS KAKENHI grant No.19K03409.


\begin{thebibliography}{1000}
\bibitem{ABD2004}
Toshiyuki Abe, Geoffrey Buhl, and Chongying Dong, 
Rationality, regularity and $C_2$-cofiniteness, 
\textit{Trans. Amer. Math. Soc.} \textbf{356} (2004), 3391--3402.

\bibitem{ADL2005}
Toshiyuki Abe, Chongying Dong, and Haisheng Li, 
Fusion rules for the vertex operator algebra $M(1)$ and $V_L^+$, 
\textit{Commun. Math. Phys.} \textbf{253} (2005), 171--219.

\bibitem{Adamovic2007}
Dra\v{z}en Adamovi\'c, 
A family of regular vertex operator algebras with two generators, 
\textit{Central European J. Math} \textbf{5} (2007), 1--18.

\bibitem{ADJR2017}
Chunrui Ai, Chongying Dong, Xiangyu Jiao, and Li Ren,
The irreducible modules and fusion rules for the parafermion vertex operator algebras, 
\textit{Trans. Amer. Math. Soc.}, published online,   
2017, https://doi.org/10.1090/tran/7302. 

\bibitem{ACL2017}
Tomoyuki Arakawa, Thomas Creutzig, and Andrew R. Linshaw, 
Cosets of Bershadsky-Polyakov algebras and rational $\mathcal{W}$-algebras of type $A$, 
\textit{Sel. Math. New. Ser.} \textbf{23} (2017), 2369--2395.

\bibitem{ALY2014}
Tomoyuki Arakawa, Ching Hung Lam, and Hiromichi Yamada, 
Zhu's algebra, $C_2$-algebra and $C_2$-cofiniteness of 
parafermion vertex operator algebras, 
\textit{Adv. Math.} \textbf{264} (2014), 261--295.

\bibitem{ALY2019}
Tomoyuki Arakawa, Ching Hung Lam, and Hiromichi Yamada,
Parafermion vertex operator algebras and W-algebras, 
\textit{Trans. Amer. Math. Soc.} \textbf{371} (2019), 4277--4301.

\bibitem{Carnahan2014}
Scott Carnahan, 
Building vertex algebras from parts, 
arXiv:1408.5215v3.

\bibitem{CKL2015}
Thomas  Creutzig, Shashank Kanade, and Andrew R. Linshaw,
Simple current extensions beyond semi-simplicity,
preprint, arXiv:1511.08754.

\bibitem{CKM2017}
Thomas Creutzig, Shashank Kanade, and Robert McRae, 
Tensor categories for vertex operator superalgebra extensions, 
preprint, arXiv:1705.05017v1.pdf

\bibitem{CL2017}
Thomas Creutzig and Andrew R. Linshaw, 
Cosets of the $\mathcal{W}^k(\mathfrak{sl}_4, f_{\textrm{subreg}})$-algebra, 
arXiv:1711.11109v1.

\bibitem{DJX2013}
Chongying Dong, Xiangyu Jiao, and Feng Xu,
Quantum dimensions and quantum Galois theory, 
\textit{Trans. Amer. Math. Soc.} \textbf{365} (2013), 6441--6469. 

\bibitem{DL1993}
Chongying Dong and James Lepowsky, 
\textit{Generalized Vertex Algebras and RelativeVertex Operators}, 
Progress in Math., Vol. 112, Birkh\"{a}user, Boston,
1993.

\bibitem{DLM2000}
Chongying Dong, Haisheng Li, and Geoffrey Mason, 
Modular invariance of trace functions in orbifold theory, 
\textit{Commun. Math. Phys.} \textbf{214} (2000), 1--56.

\bibitem{DLN2015}
Chongying Dong, Xingjun Lin, and Siu-Hung Ng, 
Congruence property in conformal field theory, 
\textit{Algebra Number Theory} \textbf{9} (2015), 2121--2166.

\bibitem{DM2004}
Chongying Dong and Geoffrey Mason, 
Rational vertex operator algebras and the effective central charge,
\textit{Internat. Math. Res. Notices} \textbf{2004}, No. 56, 2989--3008.

\bibitem{DR2017}
Chongying Dong and Li Ren,
Representations of the parafermion vertex operator algebras, 
\textit{Adv. Math.} \textbf{315} (2017), 88--101. 

\bibitem{DW2016}
Chongying Dong and Qing Wang, 
Quantum dimensions and fusion rules for parafermion vertex operator algebras, 
\textit{Proc. Amer. Math. Soc.} \textbf{144} (2016), 1483--1492.

\bibitem{vEMS2017}
Jethro van Ekeren, Sven M\"{o}ller, and Nils Scheithauer,
Construction and classification of holomorphic vertex operator algebras,
\textit{J. Reine Angew. Math.}  
published online, DOI 10.1515/crelle-2017-0046, arXiv:1507.08142.

\bibitem{FHL1993}
Igor B. Frenkel, Yi-Zhi Huang, and James Lepowsky, 
On axiomatic approaches to vertex operator algebras and modules, 
\textit{Memoirs Amer. Math. Soc.} \textbf{104}, no. 494, 1993.

\bibitem{Huang1995}
Yi-Zhi Huang, 
A theory of tensor products for module categories for a vertex operator algebra, IV 
\textit{J. Pure Appl. Algebra} \textbf{100} (1995), 173--216.

\bibitem{Huang2005}
Yi-Zhi Huang, 
Differential equations and intertwining operators, 
\textit{Commun. Contemp. Math.} \textbf{7} (2005) 375--400.

\bibitem{Huang2008a}
Yi-Zhi Huang, 
Vertex operator algebras and the Verlinde Conjecture, 
\textit{Commun. Contemp. Math.} \textbf{10} (2008), 103--154.

\bibitem{Huang2008b}
Yi-Zhi Huang, 
Rigidity and modularity of vertex tensor categories, 
\textit{Commun. Contemp. Math.} \textbf{10} (2008), 871--911.

\bibitem{HKL2015} 
Yi-Zhi Huang, Alexander Kirillov, Jr., and James Lepowsky, 
Braided tensor categories and extensions of vertex operator algebras, 
\textit{Commun. Math. Phys.} \textbf{337} (2015), 1143--1159.

\bibitem{HL1994}
Yi-Zhi Huang and James Lepowsky, 
Tensor products of modules for a vertex operator algebra and vertex tensor categories,  
in \textit{Lie Theory and Geometry, in honor of Bertram Kostant}, 
ed. by R. Brylinski, J.-L. Brylinski, V. Guillemin, V. Kac,
Birkh\"{a}user, Boston, 1994, 349--383.

\bibitem{HL1995c}
Yi-Zhi Huang and James Lepowsky, 
A theory of tensor product for module category of a vertex operator algebra, III
\textit{J. Pure Appl. Algebra} \textbf{100} (1995), 141--171.

\bibitem{KO2002} 
Alexander Kirillov, Jr. and Viktor Ostrik, 
On a $q$-analogue of the McKay correspondence and the ADE classification of 
$\widehat{\mathfrak{sl}}_2$ conformal field theories, 
\textit{Adv. Math.} \textbf{171} (2002), 183--227.

\bibitem{LY2008}
Ching Hung Lam and Hiroshi Yamauchi, 
On the structure of framed vertex operator algebras and their pointwise frame stabilizers, 
\textit{Commun. Math. Phys.} \textbf{277} (2008), 237--285.

\bibitem{LL2004}
James Lepowsky and Haisheng Li, 
\textit{Introduction to Vertex Operator Algebras and Their Representations}, 
Progress in Math., Vol. 227, Birkh\"auser, Boston,
2004.

\bibitem{Li1998}
Haisheng Li, 
An analogue of the Hom functor and a generalized nuclear democracy theorem, 
\textit{Duke Math. J.} \textbf{93} (1998), 73--114.

\bibitem{M2016}
Sven M\"{o}ller,
A cyclic orbifold theory for holomorphic vertex operator algebras and applications,
Ph.D. Thesis, Technische Universit\"{a}t Darmstadt, 2016.

\bibitem{Yamauchi2004}
Hiroshi Yamauchi, 
Module categories of simple current extensions of vertex operator algebras, 
\textit{J. Pure Appl. Algebra} \textbf{189} (2004), 315--328.
\end{thebibliography}
\end{document}